\documentclass[a4paper, 10pt]{amsart}

\usepackage{amsmath,amsthm,amssymb,graphicx,pinlabel,yhmath,tikz,listings,xcolor}
\usepackage[all]{xy}

\SelectTips{cm}{}

\allowdisplaybreaks

\usepackage{hyperref}
\hypersetup{colorlinks=true,citecolor=brown,linkcolor=brown,urlcolor=black,filecolor=black}

\theoremstyle{definition}
\newtheorem{definition}{Definition}[section]
\newtheorem{remark}[definition]{Remark}
\newtheorem{example}[definition]{Example}

\theoremstyle{plain}
\newtheorem{theorem}[definition]{Theorem}
\newtheorem{proposition}[definition]{Proposition}
\newtheorem{lemma}[definition]{Lemma}

\makeatletter
    
    \@addtoreset{equation}{section}
  \makeatother
  
 \definecolor{codegreen}{rgb}{0,0.6,0}
\definecolor{codegray}{rgb}{0.5,0.5,0.5}
\definecolor{codepurple}{rgb}{0.58,0,0.82}
\definecolor{backcolour}{rgb}{0.95,0.95,0.92}

\lstdefinestyle{mystyle}{
    backgroundcolor=\color{backcolour},   
    commentstyle=\color{codegreen},
    keywordstyle=\color{magenta},
    numberstyle=\tiny\color{codegray},
    stringstyle=\color{codepurple},
    basicstyle=\ttfamily\footnotesize,
    breakatwhitespace=false,         
    breaklines=true,                 
    captionpos=b,                    
    keepspaces=true,                 
    numbers=left,                    
    numbersep=5pt,                  
    showspaces=false,                
    showstringspaces=false,
    showtabs=false,                  
    tabsize=2
}
\newcommand{\up}{\vspace{-0.5cm}}

\newcommand{\centre}[1]{\begin{array}{c} #1 \end{array}}

\newcommand{\Q}{\mathbb{Q}}
\newcommand{\Z}{\mathbb{Z}}
\newcommand{\C}{\mathbb{C}}

\newcommand{\calA}{\mathcal{A}}
\newcommand{\calI}{\mathcal{I}}
\newcommand{\calIC}{\mathcal{IC}}
\newcommand{\calM}{\mathcal{M}}
\newcommand{\calC}{\mathcal{C}}
\newcommand{\calK}{\mathcal{K}}
\newcommand{\calKC}{\mathcal{KC}}

\newcommand{\calT}{\mathcal{T}}

\newcommand{\frakL}{\mathfrak{L}}

\newcommand{\Aut}{\operatorname{Aut}}
\newcommand{\IAut}{\operatorname{IAut}}
\newcommand{\Hom}{\operatorname{Hom}}

\newcommand{\clo}[1]{\wideparen{#1}}

\definecolor{wqwqwq}{rgb}{0,0,0}

\newcommand{\hn}[4]
{  \begin{array}{c} 
\labellist \small \hair 2pt 
\pinlabel {\scriptsize $#1$} [B] at 3 0
\pinlabel {\scriptsize $#2$} [B] at 27 0
\pinlabel {\scriptsize $#3$} [B] at 52 0
\pinlabel {\scriptsize $#4$} [B] at 78 0
\endlabellist
\includegraphics[scale=0.5]{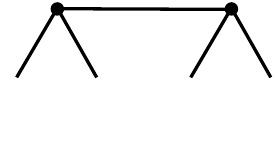}\\
\end{array}  }

\newcommand{\lsixtree}[6]{\tikz[baseline=8pt ,scale = 0.5]{
\draw [thick, color=wqwqwq] (0,1.25)-- (0,0.25);
\draw [thick,color=wqwqwq] (0,0.75)-- (3,0.75);
\draw [thick,color=wqwqwq] (3,1.25)-- (3,0.25);
\draw [thick,color=wqwqwq] (1,0.75)-- (1,1.5);
\draw [thick,color=wqwqwq] (2,0.75)-- (2,1.5);
\draw [black, fill = black] (0,0.75) circle (.5ex);
\draw [black, fill = black] (1,0.75) circle (.5ex);
\draw [black, fill = black] (2,0.75) circle (.5ex);
\draw [black, fill = black] (3,0.75) circle (.5ex);
\draw[color=wqwqwq] (0,0) node {\small $#1$};
\draw[color=wqwqwq] (0,1.6) node {\small $#2$};
\draw[color=wqwqwq] (1,1.75) node {\small $#3$};
\draw[color=wqwqwq] (2,1.75) node {\small $#4$};
\draw[color=wqwqwq] (3,1.6) node {\small $#5$};
\draw[color=wqwqwq] (3,0) node {\small $#6$};}}

\newcommand{\ltree}[4]{\tikz[baseline=8pt, scale = 0.5]{
\draw [thick,color=wqwqwq] (0,2-1.25)-- (2,2-1.25);
\draw [thick,color=wqwqwq] (0,1.5)-- (0,0);
\draw [thick,color=wqwqwq] (2,1.5)-- (2,0);
\draw [black, fill = black] (0,0.75) circle (.5ex);
\draw [black, fill = black] (2,0.75) circle (.5ex);
\draw[color=wqwqwq] (0,3-1.25) node {\small $#1$};
\draw[color=wqwqwq] (0,1.1-1.25) node {\small $#2$};
\draw[color=wqwqwq] (2,1.1-1.25) node {\small $#3$};
\draw[color=wqwqwq] (2,3-1.25) node {\small $#4$};}}

\newcommand{\lfivetree}[5]{\tikz[baseline=8pt ,scale = 0.5]{
\draw [thick,color=wqwqwq] (0,1.25)-- (0,0.25);
\draw [thick,color=wqwqwq] (0,0.75)-- (2,0.75);
\draw [thick,color=wqwqwq] (2,1.25)-- (2,0.25);
\draw [thick,color=wqwqwq] (1,0.75)-- (1,1.5);
\draw [black, fill = black] (0,0.75) circle (.5ex);
\draw [black, fill = black] (1,0.75) circle (.5ex);
\draw [black, fill = black] (2,0.75) circle (.5ex);
\draw[color=wqwqwq] (0,0) node {\small $#1$};
\draw[color=wqwqwq] (0,1.6) node {\small $#2$};
\draw[color=wqwqwq] (1,1.75) node {\small $#3$};
\draw[color=wqwqwq] (2,1.6) node {\small $#4$};
\draw[color=wqwqwq] (2,0) node {\small $#5$};}}

\newcommand{\ltritree}[3]{\tikz[baseline=8pt ,scale = 0.5]{
\draw [thick,color=wqwqwq] (1,1.5)-- (1,0.75);
\draw [thick,color=wqwqwq] (1,0.75)-- (1-0.866*0.75,0.75-.37);
\draw [thick,color=wqwqwq] (1,0.75)-- (1+0.866*0.75,0.75-.37);
\draw [black, fill = black] (1,0.75) circle (.5ex);
\draw[color=wqwqwq] (1,1.5+0.2) node {\small $#1$};
\draw[color=wqwqwq] (1-0.866*0.75-0.2,0.75-.37-0.2) node {\small $#2$};
\draw[color=wqwqwq] (1+0.866*0.75+0.2,0.75-.37-0.2) node {\small $#3$};}}

\newcommand{\treetwo}[2]{\tikz[baseline=8pt, scale = 0.5]{
\draw[thick,color=wqwqwq] (0.5,0)--(0.5,0.75);
\draw[thick,color=wqwqwq] (0.5,0.75)--(0,1.25);
\draw[thick,color=wqwqwq] (0.5,0.75)--(1,1.25);
\draw [black, fill = black] (0.5,0.75) circle (.5ex);
\draw[color=wqwqwq] (0,1.6) node {\small $#1$};
\draw[color=wqwqwq] (1,1.6) node {\small $#2$};}}

\newcommand{\treethree}[3]{\tikz[baseline=8pt, scale = 0.8]{
\draw[thick,color=wqwqwq] (0.5,0)--(0.5,0.25);
\draw[thick,color=wqwqwq] (0.5,0.25)--(0,0.75);
\draw[thick,color=wqwqwq] (0.25,0.5)--(0.5,0.75);
\draw[thick,color=wqwqwq] (0.5,0.25)--(1,0.75);
\draw [black, fill = black] (0.5,0.25) circle (.3ex);
\draw [black, fill = black] (0.25,0.5) circle (.3ex);
\draw[color=wqwqwq] (0,1) node {\small $#1$};
\draw[color=wqwqwq] (0.5,1) node {\small $#2$};
\draw[color=wqwqwq] (1,1) node {\small $#3$};}}

\newcommand{\stick}[1]{\tikz[baseline=8pt, scale = 0.4]{
\draw[thick,color=wqwqwq] (0,0)--(0,1.5);
\draw[color=wqwqwq] (0,2) node {\small $#1$};}}

\newcommand{\trait}{\!\!\begin{array}{c}{\rule{5mm}{0.4mm}}\\[0.1cm] \end{array}\!\!}

\DeclareSymbolFont{extraitalic}      {U}{zavm}{m}{it}
\DeclareMathSymbol{\stigma}{\mathord}{extraitalic}{168}

\begin{document}

\title[On the non-triviality of the torsion subgroup  of the abelianized Johnson kernel \ ]{On the non-triviality of the torsion subgroup  of the abelianized Johnson kernel}

\author{Quentin Faes}
\address{ {IMB}, Universit\'e  de Bourgogne \& CNRS, 21000 Dijon, France }
\email{{quentin.faes@u-bourgogne.fr}}

\author{Gw\'ena\"el Massuyeau}
\address{ {IMB}, Universit\'e  de Bourgogne  \& CNRS, 21000 Dijon, France }
\email{{gwenael.massuyeau@u-bourgogne.fr}}

\date{5 October 2023}

\subjclass[2010]{}
\thanks{}

\begin{abstract}
The \emph{Johnson kernel} is the subgroup of the mapping class group of a closed oriented surface that is generated by 
Dehn twists along   separating    simple closed curves. The rational abelianization of the Johnson kernel  
has been computed by Dimca, Hain and Papadima,
and a more explicit form was subsequently provided  by Morita, Sakasai and Suzuki.
Based on these results,  Nozaki, Sato and Suzuki used the theory of finite-type invariants of $3$-manifolds
to prove that   the torsion subgroup of the abelianized Johnson kernel is non-trivial.  

In this paper, we give a purely $2$-dimensional proof of the non-triviality of   this torsion subgroup and provide     a lower bound for its cardinality.
Our main tool is  the action of the mapping class group on the Malcev Lie algebra of the fundamental group of the surface.
Using the same infinitesimal techniques, we   also     provide an alternative  diagrammatic description of the rational abelianized Johnson kernel, 
and we  include  in the results the case of an oriented surface with one boundary component.
\end{abstract}

\maketitle

\vspace{-0.5cm}
 \tableofcontents

\section{Introduction} \label{sec:intro}

Let $\Sigma$ be a compact connected oriented surface with one boundary component, 
and denote by $\calM:=\calM(\Sigma)$ the mapping class group of $\Sigma$. 
Although its abelianization  is trivial  for ${g\geq 3}$~\cite{Powell},  the group $\calM$ has remarkable subgroups 
with highly non-trivial abelianization. This is particularly manifest for the \emph{Torelli group}, 
which is the subgroup $\calI:= \calI(\Sigma)$ of $\calM$ acting trivially on the homology $H:=H_1(\Sigma;\Z)$ of the surface.
In his fundamental works of the eighties (including \cite{Jo83,Jo85a,Jo85b}), 
Johnson proved that (like $\calM$) the group $\calI$ is finitely generated 
and that (unlike  $\calM$) it has an interesting abelianization $\calI_{\operatorname{ab}}$. 
In fact, Johnson gave a full characterization of  $\calI_{\operatorname{ab}}$, revealing that
its torsion-free quotient  is isomorphic to $\Lambda^3 H$
and that its  torsion subgroup is isomorphic to the space of quadratic  boolean functions on the   space    of spin structures of $\Sigma$.
The  map  $\calI \to \Lambda^3 H$  (corresponding to the canonical projection of $\calI_{\operatorname{ab}}$ onto its torsion-free quotient)
 is the first  $\tau_1$ of a series of homomorphisms $(\tau_k)_k$, 
which are now referred to as the \emph{Johnson homomorphisms} 
and are defined on the successive terms of  the \emph{Johnson filtration} $\big(\calM[k]\big)_k$.
In particular, the subgroup $\calK:=  \calM[2]$ of $\calI=\calM[1]$ plays a very important role in Johnson's works:
called the \emph{Johnson kernel}, $\calK$ is generated by Dehn twists along   separating    simple closed curves. 
  In the sequel, those generators  will be refered to as  ``separating twists''.   

Besides, Johnson did similar constructions and proved  similar results for a closed oriented surface of genus $g\geq 3$.
Here it will be convenient to think of it as the surface $\clo{\Sigma}$  obtained by gluing a $2$-disk to $\Sigma$.
Then, for simplicity,  we shall denote by $\clo{\calM}$ the mapping class group of $\clo{\Sigma}$,
by  $\clo{\calI}$ its Torelli group and by $\clo{\calK}$ its Johnson kernel.

In the last decade, major advances on the abelianization of the Johnson kernel have been accomplished for the closed surface $\clo{\Sigma}$.
Firstly, Dimca and Papadima proved that the rational abelianization $\clo{\calK}_{\operatorname{ab}}\otimes \Q$ is finite-dimensional \cite{DP}.
Later, using this result and Hain's description of the Malcev Lie algebra of $\clo{\calI}$ \cite{Hain}, Dimca, Hain and Papadima  
computed this vector space \cite{DHP}. 
More recently, Morita, Sakasai and Suzuki  \cite{MSS} could express this computation of $\clo{\calK}_{\operatorname{ab}}\otimes \Q$ 
 in a more explicit form, involving  two homomorphisms on $\clo{\calK}$ that Morita introduced in the nineties:
 namely, a by-product $\clo{d}$ of the Casson invariant  \cite{Mor91} and a ``refined'' version of the second Johnson homomorphism $\tau_2$ \cite{Mor93a}. 
As we shall see in \S \ref{sec:rational}, the results of \cite{DHP,MSS}  can  be adapted to the case of the bordered surface $\Sigma$:
this will give us the opportunity to revisit these results
and provide an alternative  diagrammatic description of ${\calK}_{\operatorname{ab}}\otimes \Q$.
 
Using the computation of $\clo{\calK}_{\operatorname{ab}}\otimes \Q$  given in \cite{MSS} and appealing to the theory of finite-type invariants of $3$-manifolds, 
Nozaki, Sato and Suzuki \cite{NSS} were able to show that the torsion subgroup of $\clo{\calK}_{\operatorname{ab}}$ is non-trivial.
To be more explicit on their methods,  let us mention that they use 
the \emph{LMO homomorphism} (which is a universal rational finite-type invariant of homology cylinders \cite{CHM,HM}),
and that their arguments involve $3$-dimensional surgery techniques (which are known as  \emph{clasper calculus} \cite{Goussarov,Habiro}). 
Hence their proof requires a certain level of expertise in the theory of finite-type invariants,
and it does not conclude with an explicit torsion element of $\clo{\calK}_{\operatorname{ab}}$.

Regarding this result  of Nozaki, Sato and Suzuki, our goal in this paper is two-fold. On the one hand,
we provide explicit elements  of the torsion subgroup of $\clo{\calK}_{\operatorname{ab}}$,
and we prove their non-triviality by purely $2$-dimensional methods.
Thus, we hope  to make their result accessible to a wider audience, 
and to open the way towards a full  computation of the torsion subgroup of~$\clo{\calK}_{\operatorname{ab}}$.
 On the other hand, we also deal with the case of the bordered surface $\Sigma$:\\

\noindent
\textbf{Theorem A.} \emph{For  a compact connected oriented surface with $0$ or $1$ boundary component, of genus $g\geq 6$,
the abelianization of the Johnson kernel has a non-trivial torsion subgroup.}\\

Our proof  of Theorem A is based on the action of $\calM$ on the Malcev Lie algebra of the fundamental group $\pi_1(\Sigma)$.
The possibility of such a proof is not so surprising, 
since Nozaki, Sato and Suzuki only use the tree-reduction of the LMO homomorphism in their arguments
and, according to \cite{Massuyeau}, the latter  encodes in some way the action of the Torelli group $\calI$  
on  the Malcev Lie algebra of  $\pi_1(\Sigma)$.
(See Remark \ref{rem:R_0} and Remark~\ref{rem:R_0_bis} for the exact relationship with the arguments of \cite{NSS}.)

The proof of Theorem A, which is done in \S \ref{sec:proofs}, can be summarized as follows.
We use the diagrammatic description of  the action of $\calI$ on the  Malcev Lie algebra of $\pi_1(\Sigma)$ that is given in~\cite{Massuyeau}.
From this \emph{infinitesimal Dehn--Nielsen representation}, 
we derive in \S \ref{sec:DN} a map $R$ from the Johnson kernel $\calK$ to a torsion abelian group.
Although this map $R$ is only polynomial of degree~$2$, 
it restricts  on the fourth term $\calM[4]$ of the Johnson filtration to a homomorphism
(which is a reduction of $\tau_4$). 
Then we exhibit an explicit element of $\calM[4]$, which is not seen by the ``core'' of the Casson invariant~$d$
but is detected by the map $R$:\\

\noindent
\textbf{Theorem B.} \emph{Assume that $g\geq 3$.
There exists a $\varphi \in \calM[4]$ such that $d(\varphi)=0$ and $R(\varphi)\neq 0$. 
Moreover, its extension $\clo{\varphi}  \in \clo{\calM}[4]$ to the closed surface $\clo{\Sigma}$ also satisfies
$\clo{d}(\clo{\varphi})=0$ and $R(\clo{\varphi})\neq 0$.}\\

Since the kernel of Morita's refinement of $\tau_2$ is $\calM[4]$, 
we deduce from Theorem B and the above-mentioned computation of $\calK_{\operatorname{ab}} \otimes \Q$
that  $\varphi$ provides a non-trivial torsion element of $\calK_{\operatorname{ab}}$, thus proving Theorem A.
Our proof of Theorem B is purely two-dimensional and involves a rather long computation of $R(\varphi)$.
We also give a second proof of Theorem B, leading to another explicit element   $\varphi'$:   
closer to the original arguments of \cite{NSS}, 
this proof is certainly less computational,   but it is done in the  $3$-dimensional  framework of homology cylinders and needs
the  techniques of clasper calculus, including results of  \cite{Goussarov,Habiro,GL,MM,CST16}.
Indeed, the restriction of $R$ to $\calM[4]$ is the reduction of $\tau_4$ that Conant, Schneiderman and Teichner
considered in \cite{CST16} under the name of ``higher-order Sato--Levine invariant'' 
(by analogy with the study of Milnor invariants of links).

We conclude the paper by giving in \S \ref{subsec:complements} lower bounds
for the cardinalities of the torsion subgroups of $\calK_{\operatorname{ab}}$ and~$\clo{\calK}_{\operatorname{ab}}$. These lower bounds
 are obtained by a rough estimation of the image of $R$, 
   using the canonical action of the mapping class group on the abelianized Johnson kernel.  \\

\noindent
\textbf{Acknowledgment.} 
The authors would like to thank Lucy Moser--Jauslin   for helpful discussions about algebraic groups.   
This research has been partly funded by the project ``ITIQ-3D'' of the Région Bourgogne Franche--Comté and the project ``AlMaRe'' (ANR-19-CE40-0001-01).
 The IMB receives support from the EIPHI Graduate School (ANR-17-EURE-0002).

\section{The infinitesimal Dehn--Nielsen representation and the map  $R$} \label{sec:DN}

The Dehn--Nielsen representation of the mapping class group is defined by its canonical action  on the fundamental group of the surface.
We review an infinitesimal version of the Dehn--Nielsen representation   that has been introduced in~\cite{Massuyeau}.
Then we derive from this a quadratic map  $R$,  from the abelianized Johnson kernel to a torsion abelian group.

\subsection{The space of tree diagrams} \label{subsec:trees}

We first recall what is the target of  the infinitesimal Dehn--Nielsen representation.
A \emph{tree diagram} is a finite, unitrivalent, connected, acyclic graph whose trivalent vertices are oriented
(i.e$.$ edges are cyclically ordered around each trivalent vertex), and whose univalent vertices are colored by~$H=H_1(\Sigma;\Z)$:
the former are called \emph{nodes} and the latter are called  \emph{leaves}.
For example, here is a tree diagram with 3 nodes and 5 leaves:
\begin{equation} \label{eq:tree_ex}
{\labellist \small \hair 2pt
\pinlabel {$a$} [tr] at 0 0
\pinlabel {$b$} [br] at 0 59
\pinlabel {$c$} [b] at 69 76
\pinlabel {$d$} [bl] at 138 63
\pinlabel {$e$} [tl] at 135 0
\endlabellist}
\includegraphics[scale=0.3]{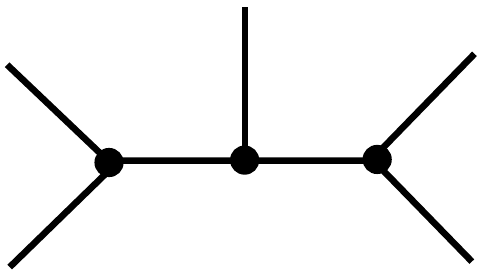}
\quad \quad \quad \quad  (\hbox{where } a,b,c,d,e \in H).
\end{equation}
\vspace{-0.3cm}

\noindent
Here, and henceforth,  orientations at trivalent vertices are always given by the trigonometric orientation of the plane.
Let $\calT(H)$
be the abelian group generated by tree diagrams modulo the following relations: \\[0.2cm]
\begin{center}
\labellist \small \hair 2pt
\pinlabel {AS} [t] at 102 -5
\pinlabel {IHX} [t] at 543 -5
\pinlabel {multilinearity} [t] at 1036 -5
\pinlabel {$= \ -$}  at 102 46
\pinlabel {$-$} at 484 46
\pinlabel {$+$} at 606 46
\pinlabel {$=0$} at 721 46 
\pinlabel {$+$} at 1106 46
\pinlabel {$=$} at 961 46
\pinlabel{$h_1+h_2$} [b] at 881 89
\pinlabel{$h_1$} [b] at 1042 89
\pinlabel{$h_2$} [b] at 1170 89
\endlabellist
\centering
\includegraphics[scale=0.28]{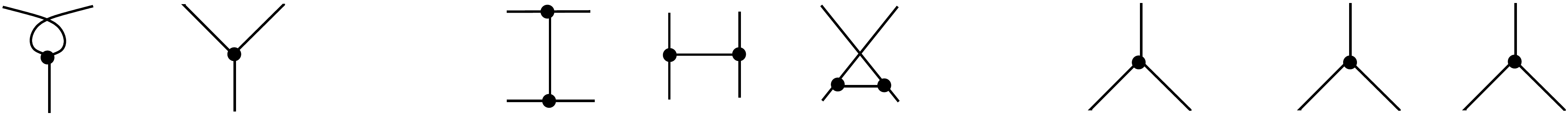}
\end{center}\vspace{0.3cm}
By defining the \emph{degree} of a tree diagram to be the number of nodes, 
we turn $\calT(H)$ into a graded abelian group:
$$
\calT(H) = \bigoplus_{d=1}^{+ \infty} \calT_d(H)
$$ 
The abelian group $\calT(H)$ has the structure of a Lie ring, which involves the intersection pairing $\omega:H \times H \to \Z$
of $\Sigma$. Specifically, the bracket of two tree diagrams $P$ and $Q$ is  given by
$$
[P,Q] :=  
\begin{array}{c} (\hbox{\small sum of all ways of $\omega$-connecting \emph{one} leaf of $P$ to \emph{one} leaf of $Q$})\end{array}
$$
where ``$\omega$-connecting'' an $x$-colored vertex  $u$ of $P$ to a $y$-colored vertex $v$ of $Q$  
results in the element of $\calT(H)$ obtained by gluing $u$ to $v$ and multiplying by $\omega(x,y)$.

Setting $H^\Q:=H_1(\Sigma;\Q)$, we define a graded $\Q$-vector space $\calT(H^\Q)$ in a similar way  by generators and relations.
Note that  we have $\calT(H^\Q) \simeq \calT(H) \otimes \Q $, and $\calT(H^\Q)$ has the structure of a Lie algebra.
This space or, to be more accurate its degree-completion  $\hat\calT(H^\Q)$, will be used in the next subsection
as the target of the ``infinitesimal'' Dehn--Nielsen representation.

The abelian group  $\calT(H)$ and the  vector space $\calT(H^\Q)$ appear in the study of Milnor invariants of links 
and Johnson homomorphisms for mapping class groups,
and they constitute the ``tree levels'' of the theories of finite-type invariants  (see for instance \cite{HM00,GL,HP,Massuyeau}).    
Consequently, their structure has been much studied.
Before reviewing their relevance for the study of mapping class groups,
we now recall what is known about this structure in relation with free Lie rings.

Let $\frakL:=\frakL(H)$ be the  Lie ring  freely generated by $H$
and, similarly, let $\frakL^\Q:=\frakL(H^\Q)$ be the Lie algebra freely generated by $H^\Q$.
For any integer $k\geq 1$, we denote by $\mathsf{D}_k(H)$ the kernel of the Lie bracket map  $H\otimes \frakL_{k+1} \to \frakL_{k+2}$
and we set $\mathsf{D}(H):= \bigoplus_k\mathsf{D}_k(H)$.
There is a homomorphism of graded abelian groups
\begin{equation} \label{eq:eta}
\eta:  \calT(H) \stackrel{}{\longrightarrow} \mathsf{D}(H)
\end{equation}
which maps any tree diagram $T$ to the sum
$$
\sum_{v}  \hbox{col}(v) \otimes \hbox{brack}(T_v)
$$
over all leaves $v$ of $T$: here $\hbox{col}(v)$ denotes the color of $v$, and $\hbox{brack}(T_v)$ is the iterated Lie bracket
defined by the tree $T$ \emph{rooted} at $v$. In the example \eqref{eq:tree_ex}, the tree rooted at its $d$-colored vertex defines\\[0.2cm]
$$
\operatorname{brack}\Bigg(\begin{array}{c}
\labellist \small \hair 2pt
\pinlabel {$e$} [b] at 2 162
\pinlabel {$a$} [b] at 111 162
\pinlabel {$b$} [b] at 167 162
\pinlabel {$c$} [b] at 222 162
\pinlabel {root} [t] at 108 5
\endlabellist
\includegraphics[scale=0.3]{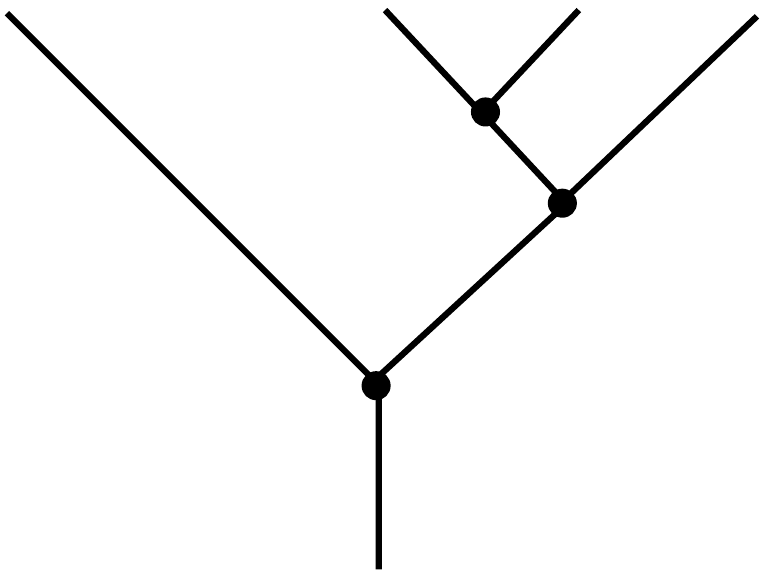} 
\end{array}\Bigg) =  [e,[[a,b],c]].
\quad
$$
The map $\eta$ is known to be an isomorphism with rational coefficients 
\begin{equation} \label{eq:etaQ}
\eta^\Q:  \calT(H^\Q) \stackrel{\simeq}{\longrightarrow} \mathsf{D}(H^\Q)
\end{equation}
(see, for instance, \cite{HP}),
but it is not bijective with   $\Z$-linear    coefficients.
Indeed $\mathsf{D}(H)$ is free abelian but, as we shall recall, $\calT(H)$ has $2$-torsion in odd degrees.

To understand the lack of bijectivity of $\eta$, Levine  considers the quasi-Lie ring $\frakL':= \frakL'(H)$
freely generated by $H$. Recall from \cite{Levine} 
that the definition of  a ``quasi-Lie'' ring  requires the bracket to be only skew-symetric, instead of alternate as a Lie bracket should be.
Levine proves that the canonical map $\frakL' \to \frakL$ is an isomorphism in odd degree and that, in even degree, there is a short exact sequence
\begin{equation} \label{eq:quasi-Lie}
0 \to \frakL_k \otimes \Z_2 \longrightarrow  \frakL'_{2k}  \longrightarrow \frakL_{2k} \to 0 
\end{equation}
where the left-hand homomorphism is defined by $x\otimes 1 \mapsto [x,x]$. 
For any $k\geq 1$, let $\mathsf{D}'_k(H)$ be the kernel of the quasi-Lie bracket $H \otimes \frakL'_{k+1} \to \frakL'_{k+2}$
and set $\mathsf{D}'(H):= \bigoplus_k\mathsf{D}'_k(H)$.
As a consequence of \eqref{eq:quasi-Lie}, there are short exact sequences
\begin{equation} \label{eq:Levine1}
0 \to  \mathsf{D}'_{2k}(H) \longrightarrow  \mathsf{D}_{2k}(H)  \longrightarrow \frakL_{k+1} \otimes \Z_2 \to 0,
\end{equation}
\begin{equation} \label{eq:Levine2}
0 \to H\otimes \frakL_{k+1} \otimes \Z_2 \longrightarrow  \mathsf{D}'_{2k+1}(H)  \longrightarrow \mathsf{D}_{2k+1}(H)  \to 0 .
\end{equation}
Levine  observes that the map $\eta$ can be defined similarly in the quasi-Lie case to get a surjective homomorphism
$$
\eta':  \calT(H) \stackrel{}{\longrightarrow} \mathsf{D}'(H),
$$
and the injectivity of $\eta'$ is proved by Conant, Schneiderman and Teichner in \cite{CST12}. Hence the   following commutative diagram:   
$$
\xymatrix{
 \calT(H) \ar[dr]_{\eta'}^-\simeq \ar[r]^{\eta} &  \mathsf{D}(H) \\
 &\mathsf{D}'(H)  \ar[u]
}
$$
Thus, the above results combine to the following statement:

\begin{theorem}[Levine,  Conant--Schneiderman--Teichner]
For any integer $k\geq 1$, there are short exact sequences
\begin{equation} \label{eq:LCST1}
0 \to  \mathcal{T}_{2k}(H) \stackrel{\eta}{\longrightarrow} \mathsf{D}_{2k}(H)   \stackrel{\varpi}\longrightarrow \frakL_{k+1} \otimes \Z_2 \to 0,
\end{equation}
\begin{equation} \label{eq:LCST2}
0 \to H\otimes \frakL_{k+1} \otimes \Z_2 \stackrel{\iota}{\longrightarrow}  \mathcal{T}_{2k+1}(H)  \stackrel{\eta}{\longrightarrow}  \mathsf{D}_{2k+1}(H)  \to 0 .
\end{equation}
\end{theorem}

\noindent
Furthermore, the homomorphism $\varpi$ is uniquely determined on the free abelian group  $ \mathsf{D}_{2k}(H)$ by the fact that 
\begin{equation} \label{eq:varpi}
\varpi\Big(  \frac 1 2  \eta(u \trait u) \Big) = \hbox{brack}(u) \otimes 1
\end{equation}
for any rooted tree $u$. (The map $\iota$ can also be explicitly defined, but we shall not need \eqref{eq:LCST2}.)
 
 \begin{remark} \label{rem:any}
 Apart from the Lie algebra structure on $\calT(H)$ which needs the symplectic form $\omega$ on $H$, 
 all  constructions and results of this subsection work for any free abelian group $H$. \hfill $\blacksquare$
 \end{remark}

\subsection{The infinitesimal Dehn--Nielsen representation} \label{subsec:infinitesimal}

The central ingredient to define the ``infinitesimal'' version of the Dehn--Nielsen representation is
the following notion.

Let $\pi:= \pi_1(\Sigma,\star)$ be the fundamental group of $\Sigma$ based at a point $\star \in \partial \Sigma$,
and let $\hat \frakL^\Q$ be the degree-completion of $\frakL^\Q$. A \emph{symplectic logansion} of $\pi$ is a map $\theta: \pi \to \hat \frakL^\Q$ 
with the following properties:
\begin{itemize}
\item for each $x\in \pi$, the Lie series $\theta(x)$ starts in degree $1$ with the class of $x$ in $\frac{\pi}{[\pi,\pi]}\otimes \Q \simeq H^\Q$;
\item for all $x,y\in \pi$, we have $\theta(xy) = \theta(x) \star \theta(y)$ where $\star$ denotes the Baker--Campbell--Hausdorff 
product\footnote{Abbreviated to \emph{BCH product} in the sequel.} in  $\hat \frakL^\Q$ induced by its Lie bracket;
\item $\theta$ maps the class $\zeta:=[\partial \Sigma]$ of the oriented boundary of $\Sigma$ to $-\omega$
where $\omega\in \Lambda^2 H^\Q \simeq \frakL^\Q_2$ is the dual of the intersection pairing.
\end{itemize}

\begin{remark}
A symplectic logansion is  a ``symplectic expansion'' in the sense of~\cite{Massuyeau},
composed  with the logarithm series to transform group-like elements into primitive elements. \hfill $\blacksquare$
\end{remark}

For concrete computations, we shall use in \S \ref{sec:proofs} the following instance of a symplectic logansion.

\begin{example} \label{ex:logansion}
Let $(\alpha_1,\dots,\alpha_g,\beta_1,\dots,\beta_g)$ be the system of ``meridians  \& parallels'' shown in Figure \ref{fig:basis},
which defines a basis of the free group $\pi$ and induces a basis $(a_1,\dots,a_g,b_1,\dots,b_g)$ of the free abelian group $H$.
According to \cite[Example 2.19]{Massuyeau},
there is a symplectic logansion $\theta$ which, in degree~$\leq 4$,  is given by
\begin{eqnarray*}
 \theta({\alpha_i}) & =&  a_i - \frac{1}{2}[a_i,b_i] + \frac{1}{12} [[a_i,b_i],b_i] - \frac{1}{2} \sum_{j<i}  [[a_j,b_j],a_i] \\
&&- \frac{1}{24} [a_i,[a_i,[a_i,b_i]]] + \frac{1}{4} \sum_{j<i} [[a_j,b_j],[a_i,b_i]] + (\deg \geq 5),\\[0.1cm]
\theta({\beta_i}) & =&  b_i - \frac{1}{2}[a_i,b_i] + \frac{1}{12} [a_i,[a_i,b_i]] + \frac{1}{4} [[a_i,b_i],b_i] 
 + \frac{1}{2} \sum_{j<i} [b_i,[a_j,b_j]] \\
&& -\frac{1}{24} [[[a_i,b_i],b_i],b_i]  + \frac{1}{4} \sum_{j<i} [[a_j,b_j],[a_i,b_i]] +  (\deg \geq 5).
\end{eqnarray*} \hfill $\blacksquare$
\end{example}

\begin{figure}
\includegraphics[scale=0.8]{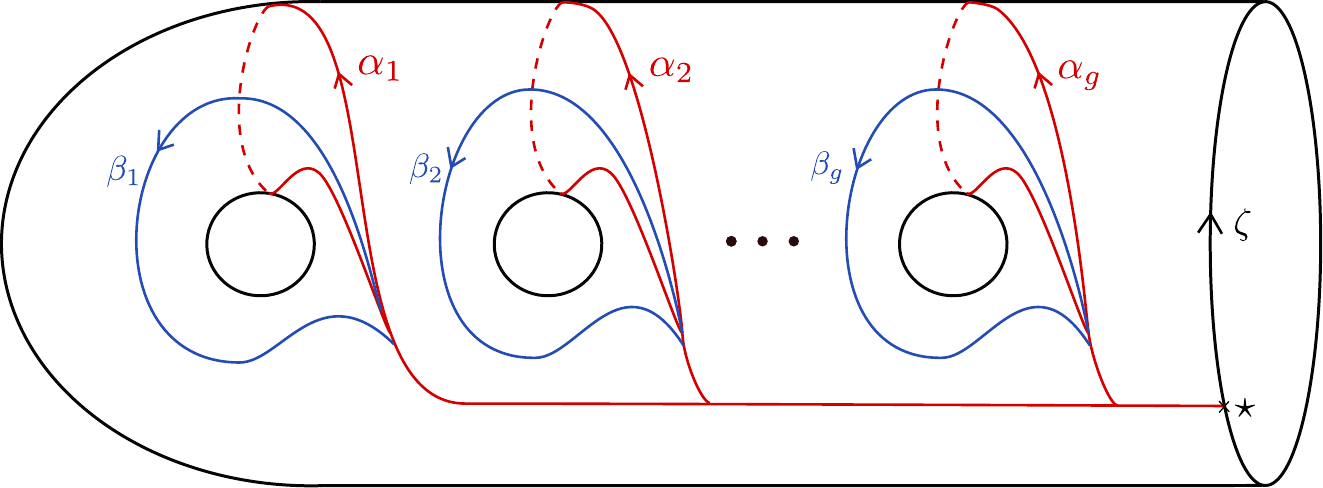}
\caption{A system of ``meridians  \& parallels'' in the oriented surface $\Sigma$}
\label{fig:basis}
\end{figure}

To pursue towards the definition of the ``infinitesimal'' version of the Dehn--Nielsen representation, we introduce the following notations.
Let  $\hbox{Aut}(\hat \frakL^\Q)$ be the group of filtered automorphisms of $\hat \frakL^\Q$, 
and let  $\hbox{Aut}_\omega(\hat \frakL^\Q)$ be the subgroup of $\hbox{Aut}(\hat \frakL^\Q)$ fixing $\omega\in \Lambda^2H^\Q \simeq \frakL_2(H^\Q)$.
As explained in \cite[\S 3.1]{Massuyeau}, the canonical action of $\calM$ on $\pi$, namely the \emph{Dehn--Nielsen representation}
$$
\rho: \calM \longrightarrow \hbox{Aut}_\zeta(\pi),
$$
is equivalent to an action of $\calM$ on the Malcev Lie algebra of $\pi$ and, via the symplectic logansion~$\theta$, 
the latter can  equivalenty be regarded as an action
$$
\varrho^\theta: \calM \longrightarrow \hbox{Aut}_\omega(\hat \frakL^\Q)
$$
of $\calM$ on the filtered Lie algebra $\hat\frakL^\Q$. Furthermore, $\varrho^\theta$ maps the Torelli group $\calI \subset \calM$
to the subgroup $\hbox{IAut}_\omega(\hat \frakL^\Q)$ of automorphisms that induce the identity at the graded level,
and this is the source of a bijection
$$
\log: \hbox{IAut}_\omega(\hat \frakL^\Q) \stackrel{\simeq}{\longrightarrow}  \hbox{Der}_\omega^+(\hat \frakL^\Q)
$$
onto the space of derivations of $\hat \frakL^\Q$ that strictly increase degrees and vanish on $\omega$.

It is well-known that $\hbox{Der}_\omega^+( \frakL)$ (resp.  $\hbox{Der}_\omega^+( \frakL^\Q)$)
 is canonically isomorphic to $\mathsf{D}(H)$ (resp. to $\mathsf{D}(H^\Q)$): 
 specifically, a derivation $d$ restricts to a homomorphism $H \to \frakL$
 and, so, induces an element of $H \otimes \frakL$ using the isomorphism $H \simeq H^*$ defined by $h\mapsto\omega(h,-)$;
then the symplectic condition $d(\omega)=0$ implies that this element of $H \otimes \frakL$ belongs to the subgroup $\mathsf{D}(H)$.
Thus, in the sequel, we will use interchangeably  $\hbox{Der}_\omega^+( \frakL)$  
(resp.  $\hbox{Der}_\omega^+( \frakL^\Q)$) and $\mathsf{D}(H)$  (resp.  $\mathsf{D}(H^\Q)$).
Recall that the space of derivations has a canonical Lie bracket given by $[d_1,d_2] := d_1 \circ d_2-d_2 \circ d_1$.
Clearly  $\hbox{Der}_\omega^+( \frakL^\Q)$ is a Lie subalgebra, 
and  it turns out that the isomorphism \eqref{eq:etaQ} preserves the Lie brackets.

Finally, we define the \emph{infinitesimal Dehn--Nielsen representation}
of the Torelli group by the following composition of maps:
\begin{equation} \label{eq:composition}
\xymatrix{
 \calI \ar[r]^-{\varrho^\theta}  \ar@/_2pc/@{-->}[rrr]^-{r^\theta}  
& \IAut_{\omega}\big(\hat\frakL^\Q) \ar[r]^-\log_-\simeq  
&  \hbox{Der}_\omega^+(\hat \frakL^\Q)  \ar[r]^-{(\eta^\Q)^{-1}}_-\simeq  & \hat \calT(H^\Q).
}
\end{equation}
The map $r^\theta$ is a group homomorphism if we endow the target $ \hat \calT(H^\Q)$  with the BCH product $\star$ induced by its Lie bracket:
thus, we have
\begin{eqnarray} 
\notag r^\theta(fh)  &= & r^\theta(f)  \star  r^\theta(h) \\
\notag &=&  r^\theta(f) + r^\theta(h) +\frac{1}{2}\big[ r^\theta(f) ,r^\theta(h) \big]  \\
\label{eq:BCH} && + \frac{1}{12}\big[ r^\theta(f),\big[ r^\theta(f),r^\theta(h) \big] \big]+ \frac{1}{12}\big[r^\theta(h) ,\big[r^\theta(h) , r^\theta(f)\big]\big] + \cdots
\end{eqnarray}
for any $f,h\in \calI$.
See \cite{Massuyeau}  for details about the above construction, and see  \cite{HM} for a survey.

\begin{remark}
The map $r^\theta$ can be extended to the full mapping class group by setting
\begin{equation} \label{eq:tilde_r}
\forall f \in \calM, \quad  \tilde{r}^\theta(f) := (\eta^\Q)^{-1}\log\big(\varrho^\theta(f) \circ f_*^{-1}\big),
\end{equation}
where $f_*:H^\Q \to H^\Q$ is the automorphism induced by $f$ in homology 
and  is extended to an automorphism of $\hat \frakL^\Q$   in the canonical way.   
Then the map $\tilde r^\theta:\calM \to  \hat \calT(H^\Q)$ satisfies
\begin{equation} \label{eq:star_tilde}
\forall f,h \in \calM, \quad \tilde r^\theta(fh) =   \tilde r^\theta(f)  \star \big(f_* \cdot \tilde r^\theta(h)\big)
\end{equation}
where $f_*$ acts on $\tilde r^\theta(h)$ by transforming the colors of all its leaves. \hfill $\blacksquare$
\end{remark}

\subsection{Truncations of the infinitesimal Dehn--Nielsen representation} \label{subsec:truncations}

We now consider the \emph{Johnson filtration} of the mapping class group
$$
\calM \supset \underbrace{\calM[1]}_{=\calI } \supset \underbrace{\calM[2]}_{=\calK } \supset \cdots \supset \calM[k] \supset \calM[k+1] \supset \cdots
$$
where $\calM[k]$ is defined as the kernel of the composition 
$$
 \xymatrix{
 \calM \ar[r]^-\rho \ar@/_2pc/@{-->}[rr]^-{\rho_k}   & \Aut(\pi) \ar[r] & \Aut(\pi/\Gamma_{k+1}\pi).
 }
 $$
(Here, and in the sequel, we denote by $G=\Gamma_1 G \supset \Gamma_2 G \supset \cdots $ the lower central series of a group~$G$.)
Recall from \cite{Mor93a} that, for any $k\geq 1$, the restriction of $\rho_{k+1}$ to $\calM[k]$ can be turned into a map
$$
\tau_k:  \calM[k] \longrightarrow \mathsf{D}_k(H) \subset H \otimes \frakL_{k+1}
$$
which is known as the \emph{$k$-th Johnson homomorphism}.
Since the abelian group  $\mathsf{D}_k(H)$ is torsion-free, we do not loose any information by considering $\tau_k$ 
with values in the vector space $\mathsf{D}_k(H^\Q)\simeq \mathsf{D}_k(H) \otimes \Q $. 
Then it is equivalent to the degree $k$ part of the infinitesimal Dehn--Nielsen representation $r^\theta$, in the sense that
\begin{equation} \label{eq:DNJ}
(\eta_\Q)^{-1} \circ \tau_k =   r_k^\theta: \calM[k]\longrightarrow  \mathcal{T}_k(H^\Q).
\end{equation}
In fact, a stronger version of \eqref{eq:DNJ} is known: the truncation 
\begin{equation} \label{eq:k-th_Morita}
r_{[k,2k[}^\theta:  \calM[k] \longrightarrow \bigoplus_{d=k}^{2k-1} \calT_d(H^\Q)
\end{equation}
of $r^\theta$  on  $\calM[k]$  to the degrees $k,k+1,\dots, 2k-1$,
which encodes the restriction of $\rho_{2k}$ to $\calM[k]$, 
 is equivalent  to Morita's ``refinement'' $\tilde \tau_k$ of $\tau_k$  \cite{Mor93a}: 
see \cite{Massuyeau} for a precise statement and a proof.

We are also interested in larger truncations of the infinitesimal Dehn--Nielsen representation on $\calM[k]$.
Specifically, for any integer $k\geq1$,  we consider the map
$$
r_{[k,3k[}^\theta: \calM[k] \longrightarrow \bigoplus_{d=k}^{3k-1} \calT_d(H^\Q)
$$
which is not a homomorphism anymore.

\begin{lemma}  \label{lem:truncation}
For all $f,h\in \calM[k]$, we have
$$
r_{[k,3k[}^\theta(fh) = r_{[k,3k[}^\theta(f)+r_{[k,3k[}^\theta(h) 
+\frac{1}{2}\sum_{\substack{i\geq k,\, j\geq k \\ i+j<3k}} \big[ r_i^\theta(f) ,r_j^\theta(h)  \big].
$$
\end{lemma}

\begin{proof}
Since $r^\theta(f)$ and $r^\theta(h)$ start in degree $k$, we deduce from \eqref{eq:BCH} that 
$$
r^\theta(fh) = r^\theta(f) +  r^\theta(h) + \frac 1 2 \big[ r^\theta(f) , r^\theta(h)] + (\deg \geq 3k)
$$
and the conclusion follows.
\end{proof}

By specializing Lemma \ref{lem:truncation} to $k=2$, we obtain for any $f,h\in \calK$
$$
r^\theta_{[2,5]}(fh)= r^\theta_{[2,5]}(f)+r^\theta_{[2,5]}(h)+\frac 1 2 \big[ r^\theta_2(f), r^\theta_{2}(h)\big] + 
\frac 1 2 \big[ r^\theta_2(f), r^\theta_{3}(h)\big] + \frac 1 2 \big[ r^\theta_3(f), r^\theta_{2}(h)\big]
$$
and, in particular,
\begin{equation}
\label{eq:trunc}
r^\theta_{4}(fh)= r^\theta_{4}(f)+r^\theta_{4}(h) +\frac 1 2 \big[( \eta^\Q)^{-1}\tau_2(f) ,( \eta^\Q)^{-1} \tau_2(h) \big].
\end{equation}

\begin{remark}
In the sequel, we will not mention anymore the isomorphism $\eta^\Q$ when composed with a Johnson homomorphism.
Said differently,  the values of  the Johnson homomorphisms are considered either as derivations or
$\Q$-linear combinations of tree diagrams, depending on the context. For instance, 
equation \eqref{eq:trunc} simply  writes $r^\theta_{4}(fh)= r^\theta_{4}(f)+r^\theta_{4}(h) +\frac 1 2 \big[\tau_2(f) , \tau_2(h) \big]$
with this convention. \hfill $\blacksquare$
\end{remark}

\subsection{The quadratic map $R$ in the bordered case}

Let $\theta$ be a symplectic logansion of $\pi$. We consider the map 
\begin{equation} \label{eq:R}
R^\theta: \calK \longrightarrow \frac{\calT_4(H^\Q)}{\calT_4(H)}, \ f \longmapsto \big(r_4^\theta(f) \!\!\mod 1 \big),
\end{equation}
and we simply denote it by $R$ in the sequel.
Here,  $\calT_4(H)$ is viewed  as a lattice in $\calT_4(H^\Q)$ (since it is torsion-free as a consequence of \eqref{eq:LCST1}),
and we refer to the congruence relation modulo $\calT_4(H)$  in $ \calT_4(H^\Q)$   as the congruence \emph{modulo $1$}.

\begin{lemma} \label{lem:R_map}
The map $R$ induces a map $R_{\operatorname{ab}}$ on the abelianization $\calK_{\operatorname{ab}} =\calK/[\calK,\calK]$, which is polynomial of degree $2$.
\end{lemma}

\begin{proof}
For any $f \in \calK$ and $h\in [\calK,\calK]$, we obtain from \eqref{eq:trunc} and the nullity of $\tau_2$ on $[\calK,\calK]$ that
$r_4^\theta(fh) = r_4^\theta(f) + r_4^\theta(h)$. Hence, to deduce that $R(f)=R(fh)$, 
it suffices to check that $R(h)=0$ and we can assume without loss of generality that $h$ is a single commutator:
$h=[h',h'']$ with $h',h''\in \calK$. A straightforward computation, still based on  \eqref{eq:trunc}, gives
$$
r_4^\theta(h)= r_4^\theta([h',h'']) =[\tau_2(h'),\tau_2(h'') ] \in \calT_4(H^\Q).
$$
Since we have $\tau_2(\calK) \subset \mathsf{D}_2(H)$,  it suffices to prove the following inclusion
in   $\calT_4(H^\Q)\simeq \mathsf{D}_4(H) \otimes \Q$:
\begin{equation} \label{eq:DD}
 \left[\mathsf{D}_2(H), \mathsf{D}_2(H) \right] \subset \calT_4(H).
\end{equation}
On this purpose, we  decompose the Lie bracket of $\calT(H^\Q)$ as follows.
Choose a symplectic basis $(a_1,\dots,a_g,b_1,\dots,b_g)$ of $H$, 
and let  $\ell: H^\Q \times H^\Q \to H^\Q$ be the bilinear map defined by
$$
\ell(a_i,b_j):=\delta_{ij}, \quad \ell(b_j,a_i):= 0, \quad \ell(a_i,a_j):=0 , \quad  \ell(b_i,b_j):=0.
$$
Given any two trees $P,Q\in \calT(H^\Q)$, we set 
$$
P \triangleright Q:= 
\begin{array}{c} (\hbox{\small sum of all ways of  $\ell$-connecting \emph{one} leaf of $P$ to \emph{one} leaf of $Q$})\end{array}
$$
and we extend this to a bilinear map $\triangleright: \calT(H^\Q) \times \calT(H^\Q) \to \calT(H^\Q)$.
Then, we have $[P,Q]= P\triangleright Q-Q\triangleright P$ since $\omega(x,y) = \ell(x,y)-\ell(y,x)$ for any $x,y\in H^\Q$.
Thus, \eqref{eq:DD} will follow from the following inclusion
in   $\calT_4(H^\Q)\simeq \mathsf{D}_4(H) \otimes \Q$:
\begin{equation} \label{eq:DDD}
 \mathsf{D}_2(H) \triangleright \mathsf{D}_2(H)  \subset \calT_4(H).
\end{equation}
It follows from \eqref{eq:LCST1} that  $\mathsf{D}_2(H)$ is  generated by the following elements:
$$
\hn{a}{b}{c}{d} \ \hbox{with } a,b,c,d\in H, \qquad \frac{1}{2} \hn{u}{v}{u}{v} \ \hbox{with } u,v\in H.
$$
Hence, to prove \eqref{eq:DDD}, it suffices to verify the following:
\begin{enumerate}
\item[(1)] $\hn{a}{b}{c}{d} \triangleright \hn{a'}{b'}{c'}{d'} \in \calT_4(H)$, for any $a,b,c,d,a',b',c',d'\in H$;
\item[(2)] $\hn{a}{b}{c}{d} \triangleright \hn{u}{v}{u}{v} \in 2\calT_4(H)$, for any $a,b,c,d,u,v\in H$;
\item[(2')] $ \hn{u}{v}{u}{v} \triangleright \hn{a}{b}{c}{d} \in 2\calT_4(H)$, for any $a,b,c,d,u,v\in H$;
\item[(3)]  $ \hn{u}{v}{u}{v} \triangleright \hn{u'}{v'}{u'}{v'} \in 4\calT_4(H)$, for any $u,v,u',v'\in H$.
\end{enumerate}
(1) is obvious. (2) (resp$.$ (2')) follows from the fact that each term resulting from the $\triangleright$ operation here is repeated due to the symmetry
of the right-hand (resp$.$ left-hand) tree diagram.  (3) is verified in a similar way.

Thus we have shown that $R$ factorizes to a map $R_{\operatorname{ab}}: \calK_{\operatorname{ab}}  \to \calT(H^\Q)/\calT(H)$ verifying
$$
R_{\operatorname{ab}}(\{f\} + \{h\}) =   R_{\operatorname{ab}}(\{f\}) +  R_{\operatorname{ab}}(\{h\}) + \frac 1 2 \big[\overline{\tau_2}(f) , \overline{\tau_2}(h) \big]
$$
where $\overline{\tau_2}:  \calK_{\operatorname{ab}}  \to  \mathsf{D}_2(H)$ is the group homomorphism induced by $\tau_2$.
The bilinearity of the Lie bracket in $\calT(H^\Q)$ implies that $R_{\operatorname{ab}}$ is a polynomial map of degree $2$ on the group $ \calK_{\operatorname{ab}} $.
\end{proof}

\begin{remark} \label{rem:R_0}
The map $R_{\operatorname{ab}}$ is \emph{not} a group homomorphism,
 i.e$.$ it is \emph{not} of degree $1$ as a polynomial map on the group $ \calK_{\operatorname{ab}}$.
Yet, using the operation $ \triangleright$ introduced in the proof of Lemma~\ref{lem:R_map},
we can instead of $R$ consider the map $R_\circ$ defined by
\begin{equation} \label{eq:R_0}
R_\circ (f) :=  \Big( r_4^\theta(f)-\frac{1}{2} \tau_2(f) \triangleright  \tau_2(f) \!\! \mod 1 \Big) \in \frac{\calT_4(H^\Q)}{\calT_4(H)},
\end{equation}
and  deduce from \eqref{eq:DDD} that $R_\circ$ is a group homomorphism on $\calK$. 
(Note that $R_\circ$ depends on the choice  of both a symplectic logansion of $\pi$ and a symplectic basis of $H$.)

When $\theta$ is the symplectic logansion defined by the LMO functor \cite[\S 5.2]{Massuyeau},
the homomorphism $R_\circ$  is equal to  the degree $4$ part of the ``mod $1$ tree reduction'' of $\log_\sqcup \widetilde{Z}^Y$, 
which is considered by Nozaki, Sato and Suzuki \cite{NSS}. 
(This equality is a consequence of \cite[Theorem 5.13]{Massuyeau}.) \hfill $\blacksquare$
\end{remark}

The next lemma shows that  the restriction of $R$ to $ \calM[4]$  is determined by 
the $4$-th Johnson homomorphism (and, so, does not depend on the choice of $\theta$).

\begin{lemma} \label{lem:R_M4}
We have the following commutative diagram
$$
\xymatrix{
\calM[4] \ar[r]^-{\tau_4} \ar[rrd]_R & \mathsf{D}_4(H) \ar@{->>}[r]^-\varpi & \frakL_3 \otimes \Z_2 \ar@{>->}[d]^-j \\
&&  \frac{\calT_4(H^\Q)}{\calT_4(H)},
}
$$
where $\varpi$ is the map given by \eqref{eq:varpi} and $j$ is  defined by
$$
j( [a,[b,c]]\otimes 1) :=  \frac 1 2 \quad
\begin{array}{c} {\labellist \small \hair 2pt
\pinlabel {$a$} [r] at 0 0
\pinlabel {$b$} [r] at 0 20
\pinlabel {$c$} [r] at 0 40
\pinlabel {$a$} [l] at 70 40
\pinlabel {$b$} [l] at 70 20
\pinlabel {$c$} [l] at 70 0
\endlabellist}
\includegraphics[scale=0.6]{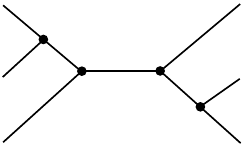} \end{array}.
$$
\end{lemma}

\begin{proof}
The isomorphism $\eta^\Q: \calT_4(H^\Q) \to \mathsf{D}_4(H^\Q)$ induces an isomorphism
between $\calT_4(H^\Q)/\calT_4(H)$ and  $\mathsf{D}_4(H^\Q)/\eta(\calT_4(H))$. 
The latter contains $\mathsf{D}_4(H)/\eta(\calT_4(H))$ which, according to \eqref{eq:LCST1}, 
is isomorphic to $\frakL_3 \otimes \Z_2$.
Then, we deduce from \eqref{eq:R} and \eqref{eq:DNJ} that the restriction of $R$ to $\calM[4]$ corresponds (through $\eta^\Q$)
to the composition 
$$
\xymatrix{
\calM[4] \ar[r]^-{\tau_4} & \mathsf{D}_4(H)  \ar@{->>}[r] & \frac{\mathsf{D}_4(H)}{\eta(\calT_4(H))}.
}
$$
Then we conclude thanks to the definition \eqref{eq:varpi} of $\varpi$.
\end{proof}

\subsection{The quadratic map $R$ in the closed case}

We now consider   the    closed surface $\clo{\Sigma}$, which is obtained from $\Sigma$ by gluing a $2$-disk.
All the previous constructions of this section for $\Sigma$ can be performed for $\clo{\Sigma}$, but with extra technical difficulties which we outline. 

First of all, we need to consider the subgroup $I$ of $\calT(H)$ that is generated by trees showing an \emph{$\omega$-vertex}, 
as shown below:
$$
\begin{array}{c}
\labellist \small \hair 2pt
\pinlabel {${\displaystyle := \quad \sum_{i=1}^g}$} at 98 31 
\pinlabel {$\omega$} [t] at 3 0
\pinlabel {$b_i$} [t] at 166 0 
\pinlabel {$a_i$} [t] at 202 0
\endlabellist
\centering
\includegraphics[scale=0.5]{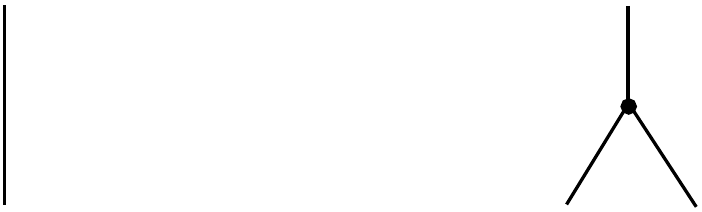}
\end{array}\\[0.2cm]
$$
It is easily deduced from the IHX relation that $I$ is an ideal of the Lie ring $\calT(H)$ (see \cite[\S 7]{HM09}), hence the quotient
$$
\clo{\calT}(H) := \frac{\calT(H)}{I}
$$
is a Lie ring. 
Besides, let $\clo{\frakL}$ be the quotient of $\frakL$ by the ideal generated by $\omega \in \frakL_2$,
 let $\clo{\mathsf{D}}_d(H)$ be the kernel of the Lie bracket map $H \otimes \clo{\frakL}_{d+1} \to \clo{\frakL}_{d+2}$ for any $d\geq 1$,
and  set 
$$
\mathsf{O} \clo{\mathsf{D}}_d(H) :=  \frac{{\clo{\mathsf{D}}}_d(H)}{(\hbox{id}_H\otimes [\cdot,\cdot])\big(\omega \otimes \clo{\frakL}_{d})}
$$
where $[\cdot,\cdot]$ denotes the Lie bracket map $H \otimes \clo{\frakL}_{d} \to \clo{\frakL}_{d+1}$.
Then, similarly to  \eqref{eq:eta}, we have  a homomorphism
\begin{equation}    \label{eq:eta_bis}
\eta: \clo{\calT}_d(H) \longrightarrow \mathsf{O} \clo{\mathsf{D}}_d(H)
\end{equation}
which, just as \eqref{eq:eta}, induces an isomorphism $\eta^\Q$ for rational coefficients.

In the case of the closed surface $\clo{\Sigma}$, the \emph{infinitesimal Dehn--Nielsen representation} of the Torelli group  \cite{HM}
is the composition 
\begin{equation} \label{eq:DN_closed}
\xymatrix{
 \clo{\calI} \ar[r]^-{\varrho^\theta}  \ar@/_2pc/@{-->}[rrr]^-{r^\theta}  
& \hbox{IOut}\big(\clo{\frakL}^\Q\big) \ar[r]^-\log_-\simeq  
&  \hbox{ODer}^+\big(\clo{\frakL}^\Q\big)  \ar[r]^-{(\eta^\Q)^{-1}}_-\simeq  & \clo{\calT}(H^\Q)
}
\end{equation}
where, to simplify notations, we have omitted the $\,\widehat \cdot\,$ decoration indicating degree-completions,
$\hbox{IOut}\big(\clo{\frakL}^\Q\big)$ is the group of automorphisms of $\clo{\frakL}^\Q$  (that induce the identity at the graded level) modulo inner automorphisms,
and $ \hbox{ODer}^+\big(\clo{\frakL}^\Q\big)$ is the Lie algebra of derivations of $\clo{\frakL}^\Q$ (that strictly increase degrees) modulo inner derivations.
In the definition of \eqref{eq:DN_closed}, we use the fact that the symplectic logansion $\theta$ induces an isomorphism between the Malcev Lie algebra of $\pi_1(\clo{\Sigma},\star)$
and (the degree-completion of) $\clo{\frakL}^\Q$ \cite[Prop$.$ 2.18]{Massuyeau},
and we have identified $\mathsf{O} \clo{\mathsf{D}}(H^\Q)$ with $\hbox{ODer}^+\big(\clo{\frakL}^\Q\big)$
in the same way we have already identified ${\mathsf{D}}(H^\Q)$ with $\hbox{Der}_\omega^+(\frakL^\Q)$.
By construction, there is a commutative diagram
\begin{equation}\label{eq:rr}
\xymatrix{
\calI \ar[r]^-{r^\theta}  \ar@{->>}[d]  &  {\calT}(H^\Q) \ar@{->>}[d] \\
\clo{\calI} \ar[r]_-{r^\theta} &  \clo{\calT}(H^\Q).
}
\end{equation}

According to \cite{Mor96}, and similarly to the case of the bordered surface $\Sigma$, 
the mapping class group $\clo{\calM}$ of $\clo{\Sigma}$ has a \emph{Johnson filtration}
$$
\clo{\calM} \supset \underbrace{\clo\calM[1]}_{=\clo\calI } \supset \underbrace{\clo \calM[2]}_{=\clo\calK } \supset \cdots \supset \clo\calM[k] \supset \clo\calM[k+1] \supset \cdots
$$
and, for every integer $k\geq 1$, the \emph{$k$-th Johnson homomorphism} takes the form
$$
\tau_k:  \clo\calM[k] \longrightarrow \mathsf{O}\clo{\mathsf{D}}_k(H) 
\subset \frac{H \otimes \clo\frakL_{k+1}}{(\hbox{id}_H\otimes [\cdot,\cdot])\big(\omega \otimes \clo{\frakL}_{k})}.
$$
According to \cite[Lemma 2]{Asada}, the abelian group $\mathsf{O}\clo{\mathsf{D}}_k(H)$  is torsion-free;
hence, as in the case of $\Sigma$, there is no loss of information in tensoring $\tau_k$ with $\Q$
and the analogue of \eqref{eq:DNJ} holds true in the closed case too.
As an analogue of Morita's refinement~$\tilde \tau_k$ of $\tau_k$ in the closed case, we shall consider the homomorphism
\begin{equation} \label{eq:k-th_Morita_closed}
r_{[k,2k[}^\theta:  \clo{\calM}[k] \longrightarrow \bigoplus_{d=k}^{2k-1} \clo{\calT}_d(H^\Q).
\end{equation}

Similarly  to the case of the surface $\Sigma$, we now define by the same formula \eqref{eq:R} a map
\begin{equation} \label{eq:R_closed}
R: \clo\calK \longrightarrow \frac{\clo\calT_4(H^\Q)}{\clo\calT_4(H)}.
\end{equation}
Here ${\clo\calT_4(H^\Q)}/{\clo\calT_4(H)}$ denotes the quotient of $\clo\calT_4(H^\Q)$ 
by the image  of  the canonical group homomorphism $\clo\calT_4(H) \to \clo\calT_4(H^\Q)$:
since it is not clear whether $\clo\calT_4(H)$ is torsion-free (in contrast to the case of $\Sigma$),
this could be a slight abuse of notation.
As a consequence of \eqref{eq:rr}, we have the following commutative diagram:
\begin{equation} \label{eq:RR}
\xymatrix{
\calK \ar[r]^-{R}  \ar@{->>}[d]  & \frac{\calT_4(H^\Q)}{\calT_4(H)}\ar@{->>}[d] \\
\clo{\calK} \ar[r]_-{R} &  \frac{\clo\calT_4(H^\Q)}{\clo\calT_4(H)}
}
\end{equation}
Hence, by Lemma \ref{lem:R_map}, we obtain  a map $R_{\operatorname{ab}}$ on the abelianization $\clo{\calK}_{\operatorname{ab}} =\clo{\calK}/[\clo{\calK},\clo{\calK}]$
which is polynomial of degree $2$.

\begin{remark} \label{rem:R_0_bis}
The variation $R_\circ: \calK \to {\calT_4(H^\Q)}/{\calT_4(H)}$ of $R$ introduced in Remark \ref{rem:R_0} does \emph{not} factorize
to a group homomorphism from $\clo{\calK}$ to  ${\clo{\calT}_4(H^\Q)}/{\clo{\calT}_4(H)}$. 
To produce from $R_\circ$ a group homomorphism on  $\clo{\calK}$, one would  need to add extra relations to the quotient  ${\clo{\calT}_4(H^\Q)}/{\clo{\calT}_4(H)}$.
These extra relations seem to be missing in the definition of the map induced by $\log_\sqcup \widetilde{Z}^Y$ 
that is considered in the proof of \cite[Cor$.$~1.5]{NSS}.
(Indeed, the subspace of diagrammatic relations needed to have the LMO homomorphism $\widetilde{Z}^Y$ 
defined on ``closed''   homology    cylinders is an ideal for the ``star product'' denoted by $\star$ in \cite{CHM},
but it is not an ideal for the disjoint union operation $\sqcup$: hence $\log_\sqcup \widetilde{Z}^Y$ does not factorize to the monoid of ``closed''   homology    cylinders.) 
\hfill $\blacksquare$
\end{remark}

The next lemma shows that  the restriction of $R$ to $\clo{\calM}[4]$  is determined by 
the $4$-th Johnson homomorphism (and, so, does not depend on the choice of $\theta$).

\begin{lemma} \label{lem:R_M4_closed}
There is a non-trivial $2$-torsion abelian group $L$ that fits into a commutative diagram of the following form:
$$
\xymatrix{
\clo\calM[4] \ar[r]^-{\tau_4} \ar[rrd]_R & \mathsf{O}\clo{\mathsf{D}}_4(H)  \ar@{->>}[r]^-\varpi
 & L  \ar@{>->}[d]^-j \\
&&  \frac{\clo{\calT}_4(H^\Q)}{\clo{\calT}_4(H)}.
}
$$
\end{lemma}

\begin{proof}
Set
$$
L:= \frac{\mathsf{O}\clo{\mathsf{D}}_4(H)}{ \eta\big(\clo{\calT}_4(H)\big)}
$$
and let $\varpi:{\mathsf{O}\clo{\mathsf{D}}_4(H)} \to L$ be the canonical projection.
Like in the bordered case, we obtain that the restriction of $R$ to $\clo\calM[4]$ is equivalent  to the composition
\begin{equation} \label{eq:ccc}
\xymatrix{
\clo\calM[4] \ar[r]^-{\tau_4} & \mathsf{O}\clo{\mathsf{D}}_4(H)  \ar@{->>}[r]^-\varpi & L 
}
\end{equation}
 through the  isomorphism 
$$
\eta^\Q :  \frac{\clo{\calT}_4(H^\Q)}{\clo{\calT}_4(H)} \stackrel{\simeq}{\longrightarrow}
 \frac{\mathsf{O}\clo{\mathsf{D}}_4(H^\Q)}{ \eta\big(\clo{\calT}_4(H)\big)}.
$$
Let $p:  \frakL \to \clo{\frakL}$ be the canonical projection, whose kernel is the ideal $ \langle\langle \omega \rangle\rangle$
generated by $\omega$. The map $H \otimes \langle\langle \omega \rangle\rangle_5 \to \langle\langle \omega \rangle\rangle_6$
defined by the Lie bracket is surjective. Hence an application of the snake lemma  to the  commutative diagram
$$
\xymatrix{
0 \ar[r] & \mathsf{D}_4(H)  \ar@{-->}[d]^-q \ar[r]  & H \otimes \frakL_5 \ar[r]  \ar[d]^-{H \otimes p_5} & \frakL_6 \ar[r] \ar[d]^-{p_6} & 0 \\ 
0 \ar[r] & \clo{\mathsf{D}}_4(H) \ar[r]  & H \otimes \clo{\frakL}_5 \ar[r] & \clo{\frakL}_6 \ar[r] & 0
}
$$
shows that the induced map $q:\mathsf{D}_4(H) \to \clo{\mathsf{D}}_4(H)$ is surjective, 
with kernel $\big(H \otimes \langle\langle \omega \rangle\rangle_5\big) \cap \mathsf{D}_4(H)$.
Hence we deduce that $q$ induces an isomorphism from
$$
 \frac{\mathsf{D_4}(H)  }{\eta(\calT_4(H))+ 
 \big(H \otimes \langle\langle \omega \rangle\rangle_5+ (\operatorname{id}\otimes [\cdot,\cdot])(\omega \otimes {\frakL}_{4})\big)\cap \mathsf{D_4}(H)}
$$
to $L$. But, by considering tree diagrams in $\calT_4(H)$ with $\omega$-vertices, we obtain the following identity of subgroups of $H \otimes \frakL_5$:
 $$
 \eta(\calT_4(H)) + \big(H \otimes \langle\langle \omega \rangle\rangle_5+ 
 (\operatorname{id}\otimes [\cdot,\cdot])(\omega \otimes {\frakL}_{4})\big)\cap \mathsf{D_4}(H) = 
 \eta(\calT_4(H)) + \big(H \otimes \langle\langle \omega \rangle\rangle_5\big) \cap \mathsf{D_4}(H).
 $$
Hence $L$ is  canonically isomorphic to 
$$
L':= \frac{\mathsf{D_4}(H)  }{ \eta(\calT_4(H)) +\big(H \otimes \langle\langle \omega \rangle\rangle_5\big) \cap \mathsf{D_4}(H) }
\qquad \Bigg( \simeq \frac{\clo{\mathsf{D}}_4(H)  }{q \eta(\calT_4(H)) } \Bigg)
$$

Being a quotient of $\mathsf{D_4}(H)/ \eta(\calT_4(H)) \stackrel{\eqref{eq:LCST1}}{\simeq} \frakL_3 \otimes \Z_2$, 
the abelian group $ L'$ is $2$-torsion. Let $A:= H/ \langle b_1,\dots,b_g\rangle$:
using  Remark \ref{rem:any}, we see that the canonical projection $H \to A$ induces a map 
$m:\mathsf{D_4}(H)/ \eta(\calT_4(H)) \to \mathsf{D_4}(A)/ \eta(\calT_4(A))$. According to \eqref{eq:LCST1},
$m$ is essentially the canonical map $ \frakL_3(H) \otimes \Z_2 \to  \frakL_3(A) \otimes \Z_2 $, so that it is not zero.
Thus, to conclude that $ L'$ is not trivial, it suffices to observe that the homomorphism $m$ factorizes through $L'$.
\end{proof}

\begin{remark} \label{rem:about_L}
It is likely that $L\simeq \clo{\frakL}_3 \otimes \Z_2$.
Indeed, the composition of homomorphisms
$$
\xymatrix{
\frakL_3 \otimes \Z_2  \ar[r]^-j &
 \frac{\calT_4(H^\Q)}{\calT_4(H)} \ar@{->>}[r] &  \frac{\clo{\calT}_4(H^\Q)}{\clo{\calT}_4(H)},
}
$$
where $j$ is the map of Lemma  \ref{lem:R_M4}, factorizes through the quotient $\clo{\frakL}_3 \otimes \Z_2 $ of $\frakL_3\otimes \Z_2 $.
Hence we have a homomorphism $\clo{j}: \clo{\frakL}_3 \otimes \Z_2  \to {\clo{\calT}_4(H^\Q)}/{\clo{\calT}_4(H)}$ such that
\begin{equation} \label{eq:parallelogramme}
\xymatrix{
& \clo{\frakL}_3 \otimes \Z_2 \ar[rr]^-{\clo{j}} \ar@{-->>}[rd]_-\exists   \ar@{->>}[ld] && \frac{\clo{\calT}_4(H^\Q)}{\clo{\calT}_4(H)}\\
 {\frakL}_3(A)  \otimes \Z_2   & & \ar[ll] L\ar@{>->}[ru]_-j&
}
\end{equation}
where the bottom map is given by the last paragraph of the proof of Lemma \ref{lem:R_M4_closed}, and the leftmost map is the canonical projection.
It is expected that  $\clo{j}$ is injective, which is equivalent to the identity
$\varpi\big((H \otimes \langle\langle \omega \rangle\rangle_5) \cap \mathsf{D_4}(H) \big)= [\omega, H] \otimes \Z_2$.\hfill $\blacksquare$
\end{remark}

%
%
\section{On the rational abelianization of the Johnson kernel} \label{sec:rational}

In this section,  we  review the computation of the rational abelianization of the Johnson kernel
 by Dimca, Hain and Papadima \cite{DHP},
in the   more     explicit form  given by Morita, Sakasai and Suzuki \cite{MSS}.
Although these results are obtained in \cite{DHP} and \cite{MSS} only for $\clo{\calK}_{\operatorname{ab}} \otimes \Q$,
we shall see that they can  be proved for  ${\calK}_{\operatorname{ab}} \otimes \Q$ too.
Furthermore, we make the  computation very  explicit by using the infinitesimal Dehn--Nielsen representation.

\subsection{The closed case}

In his study of the relationship between the Casson invariant and the structure of the Torelli group \cite{Mor89,Mor91},
Morita introduced two fundamental homomorphisms
$$
d: \calK \longrightarrow \Z  \quad \hbox{and} \quad d': \calK \longrightarrow \Z 
$$
of a very different nature. While $d'$ is defined from $\tau_2$ (and, so, is determined by the action of $\calK$ on $\pi/\Gamma_4 \pi$),
the homomorphism $d$ involves the Casson invariant in its definition (and, as a consequence of results in \cite{Hain},
it is not determined by the action of $\calK$ on $\pi/\Gamma_k \pi$ for any \emph{fixed} $k$).
Yet, both of them are invariant under the conjugacy action of $\calM$ on $\calK$, 
and they have simple values on any   separating twist $T_\gamma$     of genus $h$:
$$
d(T_\gamma)=4h(h-1)  \quad \hbox{and} \quad d'(T_\gamma)=h(2h+1).
$$
Furthermore, Morita  proved in  \cite[Theorem 5.7]{Mor91} that the linear combination 
$$
\clo{d} := - \frac{1+2g}{12} \cdot d + \frac{g-1}{3} \cdot d' 
$$
vanishes on the kernel of the canonical map $\calK \to \clo{\calK}$:
hence there is an $\clo{\calM}$-invariant group homomorphism $\clo{d}: \clo{\calK} \to \Z$ satisfying $\clo{d}(T_\gamma) = h(g-h)$.

In the case of the closed surface $\clo{\Sigma}$, the rational abelianization of the Johnson kernel is determined by the following theorem,
which results from \cite[Th$.$ B]{DHP} combined with \cite[Th$.$ 1.4]{MSS}.

\begin{theorem}[Dimca--Hain--Papadima, Morita--Sakasai--Suzuki] \label{th:DHP-MSS}
In genus $g\geq 6$, the group homomorphism
$$
\big(\clo{d},r_{[2,4[}^\theta\big): \clo{\calK} \longrightarrow \Z \oplus \clo{\calT}_2(H^\Q)  \oplus \clo{\calT}_3(H^\Q)
$$
induces a linear embedding  of $\clo{\calK}_{\operatorname{ab}} \otimes \Q$ into  $\Q  \oplus \clo{\calT}_2(H^\Q) \oplus \clo{\calT}_3(H^\Q)$.
\end{theorem}

Note that the truncation $r_{[2,4[}^\theta$ of the infinitesimal Dehn--Nielsen representation
plays the same role as the second Morita homomorphism  $\tilde \tau_2$ in the statement of \cite[Th$.$ 1.4]{MSS}.

\subsection{The bordered case}

Similarly to Theorem \ref{th:DHP-MSS}, we have the following result for the bordered  surface $\Sigma$.

\begin{theorem} \label{th:DHP-MSS_bis}
In genus $g\geq 6$, the  group homomorphism
$$
\big(d,r_{[2,4[}^\theta\big): {\calK} \longrightarrow \Z \oplus {\calT}_2(H^\Q)  \oplus {\calT}_3(H^\Q)
$$
induces a linear embedding of ${\calK}_{\operatorname{ab}} \otimes \Q$ into $ \Q  \oplus {\calT}_2(H^\Q) \oplus {\calT}_3(H^\Q)$.
\end{theorem}

To prove Theorem \ref{th:DHP-MSS_bis},
we shall mainly adapt the proof of Theorem \ref{th:DHP-MSS} given  in \cite{DHP} and~\cite{MSS}.
So, as in the closed case, our arguments will require some fundamental results of  Dimca \& Papadima \cite{DP},
 Hain \cite{Hain} (which impose the lower bound on the genus $g$) and Putman \cite{Put18}.
But, because we were not able to completely ``translate'' the arguments of \cite{DHP} from the closed case
to the bordered case, we will also need the finite generation of $\calK$ which has been obtained more recently by Ershov \& He  \cite{EH} 
and Church, Ershov \& Putman~\cite{CEP}.

Thus, as done in \cite[\S 2]{DHP}, we start with the following general situation: $G$ is a  group, 
$G_{\operatorname{abf}}:=G_{\operatorname{ab}} / \operatorname{Tors}(G_{\operatorname{ab}})$ is its torsion-free abelianization,
and $K$ is the kernel of the canonical projection $G \to G_{\operatorname{abf}}$.
We are interested in the $\Q[G_{\operatorname{abf}}]$-module
$$
K_{\operatorname{ab}}\otimes \Q
$$
where the action of $G_{\operatorname{abf}}$ on $K_{\operatorname{ab}}$ is induced by the conjugacy action of $G$ on $K$. 
The $I$-adic filtration of the group algebra $\Q[G_{\operatorname{abf}}]$ induces a filtration on $K_{\operatorname{ab}}\otimes \Q$, 
and the corresponding completion is denoted by 
$$
\widehat{K_{\operatorname{ab}}\otimes \Q}.
$$
In fact,  $\widehat{K_{\operatorname{ab}}\otimes \Q}$ has a structure of  $\widehat{\Q[G_{\operatorname{ab}}]}$-module, 
where $\widehat{\Q[G_{\operatorname{ab}}]} $ denotes the $I$-adic completion of the group algebra ${\Q[G_{\operatorname{ab}}]} $.
(Note that $\widehat{\Q[G_{\operatorname{ab}}]} $ is isomorphic to $ \widehat{\Q[G_{\operatorname{abf}}]}$,
since both groups $G_{\operatorname{ab}}$ and $G_{\operatorname{abf}}$ are abelian and they have the same rationalization.)
Similarly,  using the action of $G_{\operatorname{ab}}$ on $G'_{\operatorname{ab}}$ induced by the conjugacy action of $G$ on $G'=\Gamma_2 G$,
we can consider the completion $\widehat{G'_{\operatorname{ab}}\otimes \Q}$  of ${G'_{\operatorname{ab}}\otimes \Q}$ 
with respect to the filtration induced by the $I$-adic filtration of the group algebra $\Q[G_{\operatorname{ab}}]$. 
It is proved in \cite[Prop$.$ 2.4]{DHP}, under the assumption
\begin{equation} \label{eq:hyp1}
\hbox{``$G$ is finitely generated and $K/G'$ is finite''},
\end{equation}
 that the canonical map 
\begin{equation} \label{eq:iso1}
\widehat{G'_{\operatorname{ab}}\otimes \Q} \longrightarrow \widehat{K_{\operatorname{ab}}\otimes \Q}
\end{equation}
induced by the inclusion $G' \hookrightarrow K$ is a filtered $\widehat{\Q[G_{\operatorname{ab}}]}$-linear  isomorphism.

To go further, let us recall that any group $G$ has a \emph{Malcev completion} $\mathsf{M}(G)$
and a  \emph{Malcev Lie algebra} $\mathfrak{m}(G)$: they are defined respectively
as the group-like part and the primitive part of the complete Hopf algebra 
$\widehat{\Q[G]}$, which is the $I$-adic completion of the group algebra ${\Q[G]}$.
Recall also that the group $\mathsf{M}(G)$ and the Lie algebra $\mathfrak{m}(G)$ inherit filtrations
from $\widehat{\Q[G]}$, and that they correspond each other through the formal $\exp$ and $\log$ series.
Let  $\mathfrak{m}(G)'$ be the derived subalgebra of the complete Lie algebra  $\mathfrak{m}(G)$,
and let  $\mathfrak{m}(G)'_{\operatorname{ab}}$ be its abelianization.
The adjoint action of $\mathfrak{m}(G)$ on itself induces an action of the abelian Lie algebra
$$
\mathfrak{m}(G)/\mathfrak{m}(G)' \simeq G_{\operatorname{ab}} \otimes \Q
$$
on the vector space  $\mathfrak{m}(G)'_{\operatorname{ab}}$.
Hence $\mathfrak{m}(G)'_{\operatorname{ab}}$ has also a structure of $\widehat{S}(G_{\operatorname{ab}} \otimes \Q)$-module, 
where  $\widehat{S}(G_{\operatorname{ab}} \otimes \Q)$ is the degree-completion of the symmetric algebra ${S}(G_{\operatorname{ab}} \otimes \Q)$
generated by the vector space $G_{\operatorname{ab}} \otimes \Q$.
According to \cite[Prop$.$ 5.4]{DPS}, the canonical map $\iota: G \to \mathsf{M}(G)$ 
composed with $\log:  \mathsf{M}(G) \to  \mathfrak{m}(G)$ induces a $\widehat{\Q[G_{\operatorname{ab}}]}$-linear isomorphism
\begin{equation} \label{eq:iso2}
\widehat{G'_{\operatorname{ab}}\otimes \Q}  \longrightarrow \mathfrak{m}(G)'_{\operatorname{ab}}.
\end{equation}
Here, the complete algebra $\widehat{\Q[G_{\operatorname{ab}}]}$ is identified  
with $\widehat{S}(G_{\operatorname{ab}} \otimes \Q)$ via the expansion $G_{\operatorname{ab}}\to \widehat{S}(G_{\operatorname{ab}} \otimes \Q)$
defined by $g\mapsto \exp(g)= \sum_{i\geq 0} {g^{\otimes i }}/{i!}$.
(Note that \eqref{eq:iso2}  shifts  filtrations
by $2$ if $\mathfrak{m}(G)'_{\operatorname{ab}}$ has the filtration induced by that of  $\mathfrak{m}(G)$.)

Assume now the following formality assumption on the group $G$:
\begin{equation} \label{eq:hyp2}
\begin{array}{l}
\hbox{``There is an isomorphism of filtered Lie algebras $ \mathfrak{m}(G) \to \widehat{\operatorname{Gr}}\ \mathfrak{m}(G)$}\\
\hbox{which is the identity at the graded level.''}
\end{array}
\end{equation}
Here  $\operatorname{Gr}\mathfrak{m}(G)$ is the associated graded  of  $ \mathfrak{m}(G)$,
and $\widehat{\operatorname{Gr}}\ \mathfrak{m}(G)$ denotes its degree-completion.
We recall that  $\operatorname{Gr}\mathfrak{m}(G)$  is canonically isomorphic to the associated graded $(\operatorname{Gr} G)\otimes \Q$
of the lower central series of $G$. Thus, under the assumption \eqref{eq:hyp2}, we get an isomorphism
\begin{equation} \label{eq:iso3}
 \mathfrak{m}(G)'_{\operatorname{ab}} \longrightarrow \widehat{\mathfrak{b}}(G)
\end{equation}
where
$$
{\mathfrak{b}}(G) :=  \big(\operatorname{Gr} \mathfrak{m}(G)\big)'_{\operatorname{ab}} = 
\frac{\big(\operatorname{Gr} \mathfrak{m}(G)\big)'}{ \big[\big(\operatorname{Gr} \mathfrak{m}(G)\big)', \big(\operatorname{Gr} \mathfrak{m}(G)\big)'\big]}
$$
is the \emph{infinitesimal Alexander module} of $G$.
 
Hence, if we assume simultaneously \eqref{eq:hyp1} and \eqref{eq:hyp2}, we can compose the inverse of  the isomorphism
\eqref{eq:iso1} with \eqref{eq:iso2} and next \eqref{eq:iso3} to get a filtered $\widehat{\Q[G_{\operatorname{ab}}]}$-linear isomorphism
\begin{equation} \label{eq:iso4}
\widehat{K_{\operatorname{ab}}\otimes \Q} \longrightarrow \widehat{\mathfrak{b}}(G).
\end{equation}
We now come back to the specific situation of  the Torelli group $\calI$ of $\Sigma$.

 \begin{proof}[Proof of Theorem \ref{th:DHP-MSS_bis}]
The above discussion applies to the group $G:= \calI$.
Indeed, according to the fundamental results of Johnson \cite{Jo83,Jo85a,Jo85b},
we have $K:= \calK$ in this case  and \eqref{eq:hyp1} is satisfied for any $g\geq 3$.
Furthermore, \eqref{eq:hyp2} is satisfied too when $g\geq 3$ according to Hain \cite{Hain}. Hence, 
we get a filtered isomorphism~\eqref{eq:iso4} between $\widehat{\calK_{\operatorname{ab}}\otimes \Q}$ and $\widehat{\mathfrak{b}}(\calI)$. 

Similarly to the closed case \cite{DHP},  the most important part of the proof is the following claim:
\begin{equation} \label{eq:claim}
\hbox{``${\calK_{\operatorname{ab}}\otimes \Q}$ is a nilpotent $\Q[\calI_{\operatorname{abf}}]$-module''}
\end{equation}
 which is proved using the restricted characteristic variety of the Torelli group $\calI$. 
For any finitely-generated group $G$,  let  $\mathbb{T}_0(G):= \Hom(G_{\operatorname{abf}} , \C^*)$ and recall that
 the \emph{restricted characteristic variety} of $G$ is defined by 
 $$
 \mathcal{V}(G) := \big\{ \rho \in \mathbb{T}_0(G): H_1^\rho (G; \C) \neq 0 \big\}.
 $$
 As before, denote by $K$ the kernel of the canonical projection $G \to G_{\operatorname{abf}}$.
 It follows from a result of Dwyer \& Fried \cite{DF}, in the refined form of \cite[Cor. 6.2]{PS},
 that  $\mathcal{V}(G)$ is finite if, and only if, $H_1(K;\Q)$ is finite-dimensional.
Besides, it is now known from \cite{EH,CEP} that $\calK$ is finitely generated in genus $g\geq 4$ and, so, $H_1(\calK;\Q)$ is finite-dimensional.
(In the closed case, that $H_1(\clo{\calK};\Q)$ is finite-dimensional was proved earlier in \cite{DP}.)
Hence $\mathcal{V}(\calI)$ is finite.

To go further in the proof of the claim \eqref{eq:claim}, we next show that any element of $\mathcal{V}(\calI)$  has finite order in the group $ \mathbb{T}_0(\calI)$.
For that, we closely follow the  proof of \cite[Theorem 3.1]{DHP} but with a small variation at the end of the argument. 
(This theorem  can not be applied directly in our situation,
because $\Lambda^3 H^{\Q} \simeq \calI_{\operatorname{abf}}\otimes \Q$  is not irreducible as an $\hbox{Sp}(H^\Q)$-module.)
Let $t\in \mathcal{V}(\calI)$ and assume that it has infinite order in $\mathbb{T}_0(\calI)$.
The canonical action of the group $\hbox{Sp}(H) \simeq \calM/\calI$ on $ \mathbb{T}_0(\calI)$ leaves $\mathcal{V}(\calI)$
globally invariant. Since $ \mathcal{V}(\calI)$ is finite, the stabilizer $D_t$ of $t$ under this action is a finite-index subgroup of  $\hbox{Sp}(H)$.
By Borel's density theorem,  $\hbox{Sp}(H)$ is Zariski-dense in $\hbox{Sp}(H^{\C})$, and so is  $D_t$  in $\hbox{Sp}(H^{\C})$. 
Let $Z$ be the Zariski-closure of the subgroup of  $ \mathbb{T}_0(\calI)$ generated by~$t$.
Since the algebraic group $Z$ is infinite, its dimension is at least $1$,
and so is the dimension of its Lie algebra $T_1Z$. 
Hence,  $T_1Z$ constitutes in the tangent space $T_1  \mathbb{T}_0(\calI) = \hbox{Hom}( \Lambda^3H^\C,\C)$,
of  $\mathbb{T}_0(\calI)$ at the trivial character $1$, a subspace which is fixed by $D_t$.
Hence, by Zariski-density of $D_t$ in  $\hbox{Sp}(H^{\C})$,  we obtain that the $\hbox{Sp}(H^{\C})$-invariant part of 
$ \hbox{Hom}( \Lambda^3H^\C,\C)$ is not reduced  to zero.
But this is not possible, because we have $ \hbox{Hom}( \Lambda^3H^\C,\C) \simeq \Lambda^{2g-3} H^\C$ as  $\hbox{Sp}(H^{\C})$-modules  
and $ \Lambda^k H^\C$ has no $\hbox{Sp}(H^{\C})$-invariant part if $k$ is not even. 
We conclude that any $t \in \mathcal{V}(\calI)$ is a torsion element of the group $ \mathbb{T}_0(\calI)$.
Since $ \mathcal{V}(\calI)$ is finite, we can thus find an integer $m\geq 1$  such that
\begin{equation} \label{eq:m-torsion}
\forall t\in  \mathcal{V}(\calI), \quad t^m=1\in  \mathbb{T}_0(\calI).
\end{equation}

Next, the proof of claim \eqref{eq:claim} is strictly the same as in \cite[Proof of Th$.$ A]{DHP}
 and we only highlight the main points for the reader's convenience. Let $\calI(m)$ be the preimage by the canonical projection $\calI \to \calI_{\operatorname{abf}}$
 of the subgroup of $ \calI_{\operatorname{abf}}$ consisting of elements that are divisible by $m$,
 and let $\calK(m)$ be the kernel of the canonical projection $\calI(m) \to \calI(m)_{\operatorname{abf}}$.
 According to \cite[Th$.$ B]{Put18}, all subgroups of $\calI$ of finite index containing $\calK$ have the same first Betti numbers for $g\geq 3$.
 Hence \cite[Lemma 3.4]{DHP} can be applied to deduce that
 \begin{itemize}
 \item[(i)] the inclusion $\calI(m) \hookrightarrow \calI$ induces an isomorphism 
 $\widehat{\calI(m)'_{\operatorname{ab}} \otimes \Q}  \to \widehat{\calI'_{\operatorname{ab}} \otimes \Q}$,
 \item[(ii)] we have $\calK(m) = \calK$.
 \end{itemize}
 By a double application of \eqref{eq:iso1}, we get a  $\widehat{\Q[\calI_{\operatorname{ab}}]}$-linear isomorphism
 $\widehat{\calI'_{\operatorname{ab}} \otimes \Q} \simeq \widehat{\calK_{\operatorname{ab}} \otimes \Q}$
 and a $\widehat{\Q[\calI(m)_{\operatorname{ab}}]}$-linear isomorphism
 $\widehat{\calI(m)'_{\operatorname{ab}} \otimes \Q} \simeq \widehat{\calK(m)_{\operatorname{ab}} \otimes \Q}$.
 By the statement (ii) above, $\widehat{\calK(m)_{\operatorname{ab}} \otimes \Q}$ is the completion ${}^m\widehat{\calK_{\operatorname{ab}}\otimes \Q}$
 of ${\calK_{\operatorname{ab}}\otimes \Q}$ defined by the $I$-adic filtration of $\Q[\calI(m)_{\operatorname{abf}}]$
 via the algebra homomorphism $\Q[\calI(m)_{\operatorname{abf}}] \to \Q[\calI_{\operatorname{abf}}] $ induced by the inclusion $\calI(m) \hookrightarrow \calI$.
 But, if we identify $\calI(m)_{\operatorname{abf}}$ with $\calI_{\operatorname{abf}}$ as in the proof of  \cite[Lemma 3.4]{DHP},
 that algebra homomorphism corresponds to the homomorphism $\Q[\iota_m]:\Q[ \calI_{\operatorname{abf}}] \to \Q[\calI_{\operatorname{abf}}]$ induced by
the ``multiply by $m$'' map $\iota_m:\calI_{\operatorname{abf}} \to  \calI_{\operatorname{abf}}$.
Hence,  denoting by ${}^m(\calK_{\operatorname{ab}}\otimes \Q)$  the vector space $\calK_{\operatorname{ab}}\otimes \Q$
with the structure of $\Q[\calI_{\operatorname{abf}}]$-module defined by the algebra map 
$\Q[\iota_m]: \Q[\calI_{\operatorname{abf}}] \to \Q[\calI_{\operatorname{abf}}]$, 
we see that ${}^m\widehat{\calK_{\operatorname{ab}}\otimes \Q}$ is the completion of ${}^m(\calK_{\operatorname{ab}}\otimes \Q)$
with respect to the fitration defined by the $I$-adic filtration of $\Q[\calI_{\operatorname{abf}}]$.
We deduce from the above statement (i) that the $I$-adic completion $\widehat{\calK_{\operatorname{ab}} \otimes \Q}$ 
 of the $\Q[\calI_{\operatorname{abf}}]$-module $\calK_{\operatorname{ab}} \otimes \Q$
 is isomorphic to the $I$-adic completion  ${}^m\widehat{\calK_{\operatorname{ab}}\otimes \Q}$  
 of the   $\Q[\calI_{\operatorname{abf}}]$-module ${}^m(\calK_{\operatorname{ab}}\otimes \Q)$.
 
Next, exactly as in the second paragraph of \cite[p. 816]{DHP} (which involves  a result of \cite{PS} relating  restricted character varieties 
to supports of Alexander modules), we deduce from \eqref{eq:m-torsion} that
 ${}^m(\calK_{\operatorname{ab}}\otimes \Q)$ is nilpotent as a $\Q[\calI_{\operatorname{abf}}]$-module.
Thus the canonical map ${}^m(\calK_{\operatorname{ab}}\otimes \Q) \to {}^m\widehat{\calK_{\operatorname{ab}}\otimes \Q}$
is injective, and so is the canonical map $\calK_{\operatorname{ab}}\otimes \Q \to \widehat{\calK_{\operatorname{ab}}\otimes \Q}$.
In other words, the filtration of $\calK_{\operatorname{ab}}\otimes \Q$ defined by the $I$-adic filtration of  $\Q[\calI_{\operatorname{abf}}]$
has trivial intersection; moreover, this filtration must stabilize since $\calK_{\operatorname{ab}}\otimes \Q$ is finite-dimensional.
 Thus, we have proved claim~\eqref{eq:claim}.  (Note that  the arguments can be continued a little bit further, as in the closed case,
to deduce that  $\mathcal{V}(\calI)$ is actually reduced to the trivial character.)

Since the canonical map ${\calK_{\operatorname{ab}}\otimes \Q} \to \widehat{\calK_{\operatorname{ab}}\otimes \Q}$  
 is an isomorphism,  $\widehat{\calK_{\operatorname{ab}}\otimes \Q}$ is finite-dimensional
 so that, by \eqref{eq:iso4}, the graded module  ${\mathfrak{b}}(\calI)$ 
 is concentrated in finitely many degrees. Therefore, in fine, there is a filtered $\widehat{\Q[\calI_{\operatorname{ab}}]}$-linear isomorphism
\begin{equation} \label{eq:iso5}
\digamma:{\calK_{\operatorname{ab}}\otimes \Q} \longrightarrow{\mathfrak{b}}(\calI)
\end{equation}
(which  is induced by any isomorphism $ \mathfrak{m}(\calI) \to \widehat{\operatorname{Gr}}\ \mathfrak{m}(\calI)$ 
giving the identity at the graded level). Moreover, the least integer $\ell$ such that ${\mathfrak{b}}_k(\calI)=0$ for all $k \geq \ell$
corresponds to the least integer $\ell-2$ such that $I^{\ell-2}$ acts trivially on $\calK_{\operatorname{ab}}\otimes \Q$,
where $I$ is the augmentation ideal of $\Q[\calI_{\operatorname{ab}}]$.

 We now follow the same strategy as in \cite[Proof of Th$.$ 1.4]{MSS} 
 but, again, with some  variations with respect to the closed case. We also give more explicit arguments.
 Since the graded Lie algebra $\operatorname{Gr} \mathfrak{m}(\calI)$ is generated by its degree $1$ part, we have 
 $$
 \big(\operatorname{Gr} \mathfrak{m}(\calI)\big)' = \operatorname{Gr}_{\geq 2} \mathfrak{m}(\calI).
 $$
 By \cite[Th$.$ 1.2 \& Prop$.$ 3.1]{MSS}, the canonical map between
 $\operatorname{Gr}_4 \mathfrak{m}(\calI) \simeq (\Gamma_4 \calI/ \Gamma_5 \calI) \otimes \Q $
 and $(\calM[4]/ \calM[5])\otimes \Q$ is an isomorphism for $g\geq 6$; furthermore, \cite[Lemma 4.4 \& Th$.$ 4.7]{Sakasai} implies that
 the Lie bracket map $\Lambda^2\big( (\calM[2]/ \calM[3])\otimes \Q \big) \to (\calM[4]/ \calM[5])\otimes \Q$  is surjective  for $g\geq 4$:
 therefore, the Lie bracket map $\Lambda^2 \operatorname{Gr}_2 \mathfrak{m}(\calI) \to \operatorname{Gr}_4 \mathfrak{m}(\calI)$ is surjective for $g\geq 6$.
 It follows that the degree $4$ part of ${\mathfrak{b}}(\calI)$ is trivial. Next, it can be proved by an induction on $j\geq 4$
 that ${\mathfrak{b}}_j(\calI)=0$  using  the fact that  $\operatorname{Gr} \mathfrak{m}(G)$ is generated in degree $1$: 
 the argument is general, and the same as used in \cite[Proof of Th$.$ 1.4]{MSS} for the closed case.
(In particular, we deduce that the integer $\ell$ defined in the previous paragraph is equal to $4$.)
Thus we have obtained that ${\mathfrak{b}}(\calI)$ is concentrated in degrees $2$ and $3$:
$$
{\mathfrak{b}}(\calI) = {\mathfrak{b}}_2(\calI) \oplus {\mathfrak{b}}_3(\calI)  
= \operatorname{Gr}_2 \mathfrak{m}(G) \oplus \operatorname{Gr}_3 \mathfrak{m}(G).
$$
so that the the isomorphism \eqref{eq:iso5} can be viewed as an isomorphism:
$$
\digamma=(\digamma_2,\digamma_3): \calK_{\operatorname{ab}}\otimes \Q \longrightarrow 
\Big( \frac{\Gamma_2 \calI}{\Gamma_3 \calI} \otimes\Q \Big)\oplus \Big( \frac{\Gamma_3 \calI}{\Gamma_4 \calI} \otimes\Q  \Big)
$$

Let now $j: \calK \to \calK_{\operatorname{ab}}\otimes \Q$ be the canonical homomorphism.
The statement of Theorem \ref{th:DHP-MSS_bis} can be rephrased as follows:
\begin{equation} \label{eq:ker(j)}
\ker(j) = \ker(d) \cap \ker \big(r^\theta_{[2,4[}\big).
\end{equation}
Since $\digamma$ is an isomorphism, we have $\ker(j) =\ker (\digamma \circ j)$;
besides, since $\digamma$ arises from a formality isomorphism of $\mathfrak{m}(\calI)$ as in \eqref{eq:hyp2}, we have the equality
$\ker (\digamma \circ j)=   \sqrt{\Gamma_4 \calI}$ in $\calK$,
where $\sqrt{\Gamma_4 \calI}$ denotes the radical of $\Gamma_4 \calI$ in $\calI$. 
Hence we see that \eqref{eq:ker(j)} is equivalent to
\begin{equation} \label{eq:radical}
 \sqrt{\Gamma_4 \calI} =  \ker(d) \cap \ker \big(r^\theta_{[2,4[}\big).
\end{equation}
That $  \sqrt{\Gamma_4 \calI}$ is contained in $ \ker(d) \cap \ker \big(r^\theta_{[2,4[}\big)$ is clear, 
since we have $\Gamma_3 \calI \subset \ker(d)$ and $\calM[4] = \ker \big(r^\theta_{[2,4[}\big)$.
To prove the converse, let  $k \in \ker(d) \cap \ker \big(r^\theta_{[2,4[}\big)$. 
Recall that $[\calI ,\calI]$ is of finite index in $\calK$:
hence, to justify that $k\in \sqrt{\Gamma_4 \calI}$,
we can assume without loss of generality that $k\in \Gamma_2 \calI$.
We know from \cite{Mor89,Mor91,Hain} that 
\begin{equation} \label{eq:Morita}
(d,\tau_2): \frac{\Gamma_2 \calI}{\Gamma_3 \calI} \otimes \Q \longrightarrow \Q \oplus \calT_2(H^\Q)
\end{equation}
is an isomorphism for any $g\geq 3$. Therefore we have  $k\in  \sqrt{\Gamma_3 \calI}$:
let $m\in \mathbb{N}$ be such that $k^m \in \Gamma_3 \calI$.
We have $k^m\in \calM[3]$ and $\tau_3(k^m) = m\, r^\theta_3(k)=0$.
But we also know from \cite[Prop$.$ 6.3]{Mor99} that $\tau_3$ 
induces an embedding of  $(\Gamma_3 \calI/\Gamma_4 \calI)\otimes \Q$ into $\calT_3(H^\Q)$:
 therefore, $k^m$  belongs to $\sqrt{\Gamma_4 \calI}$, and so does $k$. This completes the proof of \eqref{eq:ker(j)}.
\end{proof}

\subsection{Complements}

Theorem \ref{th:DHP-MSS} and Theorem \ref{th:DHP-MSS_bis} produce embeddings of the rational abelianized Johnson kernel
into well-understood vector spaces, and these theorems will be enough for our purpose of proving Theorem A.
Yet, for the sake of completeness, we now identify  the  images and   the equivariance property 
of those embeddings. In the closed case, similar results  were given in \cite[Theorem B]{DHP}.
Here, we consider both the closed case and the bordered case.

The homomorphism $d: \calK \to \Z$ (resp. $\clo{d}: \clo{\calK} \to \Z$) is known to be $\calM$-invariant 
(resp$.$  $\clo{\calM}$-invariant). Thus, for the equivariance property of the embeddings of Theorem~\ref{th:DHP-MSS} and Theorem~\ref{th:DHP-MSS_bis},
we only have to understand how $r^\theta_{[2,4[}$ behaves under conjugacy by the mapping class group. 
For this, we consider the following analogue of Morita's extension of $\tau_1$~\cite{Mor93b}:
\begin{equation}
\tau_1^\theta: \calM \longrightarrow \Lambda^3 H^\Q \rtimes \operatorname{Sp}(H^\Q), \ f \longmapsto \big(\tilde{r}_1^\theta(f), f_*\big)
\end{equation}
Here $\tilde{r}_1^\theta$ denotes the degree $1$ part of the map $\tilde{r}^\theta: \calM \to \hat\calT(H^\Q)$  defined by \eqref{eq:tilde_r}.
It follows from \eqref{eq:star_tilde} that $\tau_1^\theta$ is a group homomorphism; its kernel is  $\calK$.
It can be also checked that~$\tau_1^\theta$ induces a  group homomorphism 
$\tau^\theta_1: \clo{\calM} \to \frac{\Lambda^3 H^\Q}{\omega \wedge H^\Q} \rtimes \operatorname{Sp}(H^\Q)$ whose kernel is  $\clo{\calK}$.
Besides, the target group $\Lambda^3 H^\Q \rtimes \operatorname{Sp}(H^\Q)$ acts on ${\calT}_2(H^\Q)  \oplus {\calT}_3(H^\Q)$ by
$$
\forall (w,\psi) \in \Lambda^3 H^\Q \times \operatorname{Sp}(H^\Q), \  \forall (a,b) \in {\calT}_2(H^\Q)  \oplus {\calT}_3(H^\Q),  \quad
 (w,\psi)\cdot \big(a,b\big): = (\psi\cdot a, \psi \cdot b + [w,\psi \cdot a] ),
$$
and we also have an induced action of $ \frac{\Lambda^3 H^\Q}{\omega \wedge H^\Q} \rtimes \operatorname{Sp}(H^\Q)$ 
on  $\clo{\calT}_2(H^\Q)  \oplus \clo{\calT}_3(H^\Q)$.

\begin{proposition}
The map $r^\theta_{[2,4[}: \calK \to {\calT}_2(H^\Q)  \oplus {\calT}_3(H^\Q)$ 
(resp. $r^\theta_{[2,4[}: \clo{\calK} \to \clo{\calT}_2(H^\Q)  \oplus \clo{\calT}_3(H^\Q)$) is equivariant over the group homomorphism $\tau_1^\theta$.
\end{proposition}

\begin{proof}
By the commutativity of \eqref{eq:rr}, it suffices to prove the proposition in the bordered case. 
A straightforward computation gives
\begin{eqnarray*}
\forall f\in \calM, \forall h\in \calI, \quad r^\theta(fhf^{-1}) &=&  (\eta^\Q)^{-1} \log \big( \varrho^\theta(fhf^{-1}) \big) \\
&=&  (\eta^\Q)^{-1} \log \big( (\varrho^\theta(f) f_*^{-1}) \circ (f_* \varrho^\theta( h) f_*^{-1} ) \circ  (\varrho^\theta(f) f_*^{-1})^{-1} \big) \\
&=& e^{[\tilde r^\theta(f),-]}\big(f_*\cdot r^\theta(h)) \\
&=& f_* r^\theta(h) + \big[ \tilde r^\theta(f),  f_* r^\theta(h) \big] 
+ \frac{1}{2} \big[ \tilde r^\theta(f), \big[ \tilde r^\theta(f),  f_*r^\theta(h) \big]\big] + \cdots
\end{eqnarray*}
which implies that 
$$
\forall f\in \calM, \forall h\in \calK, \quad  r^\theta_{[2,4[}(fhf^{-1}) = f_* \cdot  r^\theta_{[2,4[}(h) + \big[\tilde r_1^\theta(f),f_* \cdot r^\theta_{2}(h)\big].
$$

\up
\end{proof}

To understand the image of $r^\theta_{[2,4[}: \calK_{\operatorname{ab}}\otimes \Q \to \calT_2(H^\Q) \oplus \calT_3(H^\Q)$, we need  \emph{Morita's trace}. This linear map
$
\operatorname{Tr}_k: \calT_k(H^\Q) \to S^k(H^\Q)
$
has been introduced in \cite{Mor93a} for any odd $k$, and it is defined for $k=3$ by 
$$ 
\operatorname{Tr}_3\Bigg(\ \begin{array}{c}
{\labellist \small \hair 2pt
\pinlabel {$a$} [tr] at 0 0
\pinlabel {$b$} [br] at 0 59
\pinlabel {$c$} [b] at 69 76
\pinlabel {$d$} [bl] at 138 63
\pinlabel {$e$} [tl] at 135 0
\endlabellist}
\includegraphics[scale=0.3]{fission} \end{array} \ \Bigg) 
:= 2\omega(e,a)\, bcd + 2\omega(a,d)\, ecb+ 2 \omega(d,b)\, ace + 2 \omega(b,e)\, dca.
$$
A straightforward computation shows that it factorizes to $\clo{\operatorname{Tr}}_3: \clo{\calT}_3(H^\Q) \to S^3(H^\Q)$.

\begin{proposition}
The image of $r^\theta_{[2,4[}$ is $ \calT_2(H^\Q) \oplus \ker \operatorname{Tr}_3$
(resp$.$, $\clo{\calT}_2(H^\Q) \oplus \ker \clo{\operatorname{Tr}}_3$) in the bordered case (resp$.$, in the closed case).
\end{proposition}

\begin{proof}
By works of Morita \cite{Mor99}, we have the following isomorphisms:
$$
\frac{\calM[2]}{\calM[3]}\otimes \Q \mathop{\longrightarrow}_{\simeq}^{\tau_2} \calT_2(H^\Q)
\quad \hbox{and} \quad 
\frac{\calM[3]}{\calM[4]}\otimes \Q \mathop{\longrightarrow}_{\simeq}^{\tau_3}\ker(\operatorname{Tr}_3) \subset \calT_3(H^\Q)
$$
We deduce from these isomorphisms and the injectivity of $r^\theta_{[2,4[}:  \frac{\calM[2]}{\calM[4]}\otimes \Q \to \calT_2(H^\Q) \oplus \calT_3(H^\Q)$ that
$$
 \dim \big( \calT_2(H^\Q) \oplus \ker \operatorname{Tr}_3 \big) =\dim \Big( \frac{\calM[2]}{\calM[4]}\otimes \Q \Big) = \dim r^\theta_{[2,4[}(\calK_{\operatorname{ab}}\otimes \Q) .
$$
So, to conclude that $r^\theta_{[2,4[}(\calK_{\operatorname{ab}}\otimes \Q)$  is equal to  $\calT_2(H^\Q) \oplus \ker \operatorname{Tr}_3$, 
it is enough to recall from \cite[Prop$.$~7.3]{MS} that the former is contained in the latter.
(Note that the triviality of $\operatorname{Tr}_3 \circ \tau_3$, due to Morita \cite[Th$.$~6.11]{Mor93a}, is not enough to conclude here.)

The proof in the closed case follows the same lines.
\end{proof}

\section{Proofs of Theorem A and Theorem B}   \label{sec:proofs}

In this section, we prove Theorem A and Theorem B.  
In fact, we give two proofs of Theorem B: the first one is purely 2-dimensional, 
while the second one uses $3$-dimensional  surgery techniques.

\subsection{Proof of Theorem A}  \label{subsec:proof_A}

We first prove Theorem A assuming Theorem B.
We only deal with  the case of the bordered surface $\Sigma$,
since the case of the closed surface $\clo{\Sigma}$ is proved exactly in the same way.

Theorem B asserts the existence of an element $\varphi \in \calM[4]$ such that $d(\varphi)=0$ and $R(\varphi) \neq 0$.
That $\varphi$ belongs to $\calM[4]$ implies that $r^\theta_{[2,4[}(\varphi)=0$. 
Then we deduce from Theorem \ref{th:DHP-MSS_bis} that the class $\{\varphi\} \in \calK_{\operatorname{ab}}$ is a torsion element.
Furthermore, we have $R_{\operatorname{ab}}(\{\varphi\})= R(\varphi) \neq 0$ and we conclude that $\{\varphi\} \neq 0 \in \calK_{\operatorname{ab}}$.

\subsection{Proof of Theorem B}  \label{subsec:proof_B}

   We shall exhibit an element $\varphi \in \calM[3]$ of the form
$$
\varphi=[i,k], \quad \hbox{where  $i\in \calI$ and $k\in \calK$}.
$$
Recall that $\calI$ is generated by opposite Dehn twists $T_{c^-}^{-1} T_{c^+}$ along  pairs $(c^-,c^+)$ of simple closed curves
that cobound a subsurface of $\Sigma$:
in short,  $T_{c^-}^{-1} T_{c^+}$ is called a \emph{bounding pair map}. 
The above element $i$ will be given in this generating system of $\calI$, 
while $k$ will be given as a product of separating twists.
  
It will turn out that  $\varphi$ actually belongs to $\calM[4]$ and satisfies the following:
\begin{eqnarray}
\label{eq:dos} R(\varphi)&=& \frac{1}{2} \lsixtree{a_1}{a_2}{a_3}{a_3}{a_2}{a_1} \mod 1
\end{eqnarray} 
Since $\varphi$ belongs to $[\calM, \calK]$, we have $d(\varphi)=0$; since we have $R(\varphi) = j([a_3,[a_2,a_1]])$
and $j$ is injective, we have $R(\varphi)\neq 0$. This proves Theorem B for the bordered surface $\Sigma$.

The extension $\clo{\varphi}$ of $\varphi$ to $\clo{\Sigma}$ satisfies $\clo{d}(\clo{\varphi})=-\frac{1+2g}{12}d(\varphi)$ since $\varphi \in \calM[3]$: 
therefore  $\clo{d}(\clo{\varphi})=0$.
Besides, it follows from \eqref{eq:RR} that $R(\clo{\varphi}) {=} \clo{j}([a_3,[a_2,a_1]]\otimes 1)$
where the homomorphism $\clo{j}$ is introduced in Remark~\ref{rem:about_L}.
Since $[a_3,[a_2,a_1]]\otimes 1$ is a non-trivial element of $\frakL_3(A)\otimes \Z_2$, 
the commutativity of \eqref{eq:parallelogramme} implies that $R(\clo{\varphi}) \neq 0$.
This proves Theorem B for the closed surface~$\clo{\Sigma}$.

The rest of the subsection is devoted to the construction of $\varphi\in \calM[4]$ and the proof of~\eqref{eq:dos}.

\subsubsection{The element $\varphi$ of $\calM[4]$}

Let $(c_1^+, c_1^-)$ and $(c_2^+, c_2^-)$ be the pairs of   curves    in $\Sigma$ shown in Figure \ref{fig:BP}:
note that   $c_i^+$ and  $c_i^-$ cobound a  subsurface of genus $1$. 
Then consider the following product of  bouding pair maps:   
\begin{equation} \label{eq:i}
i := \big(T_{c_1^{-}}^{-1}\,  T_{c_1^+} \big) \circ \big(T_{c_2^{-}}^{-1}\,  T_{c_2^+} \big)^{-1} = T_{c_1^{-}}^{-1}\,  T_{c_2^{-}} \in \calI.
\end{equation}
 \begin{figure}[h]
 \includegraphics[scale=0.5]{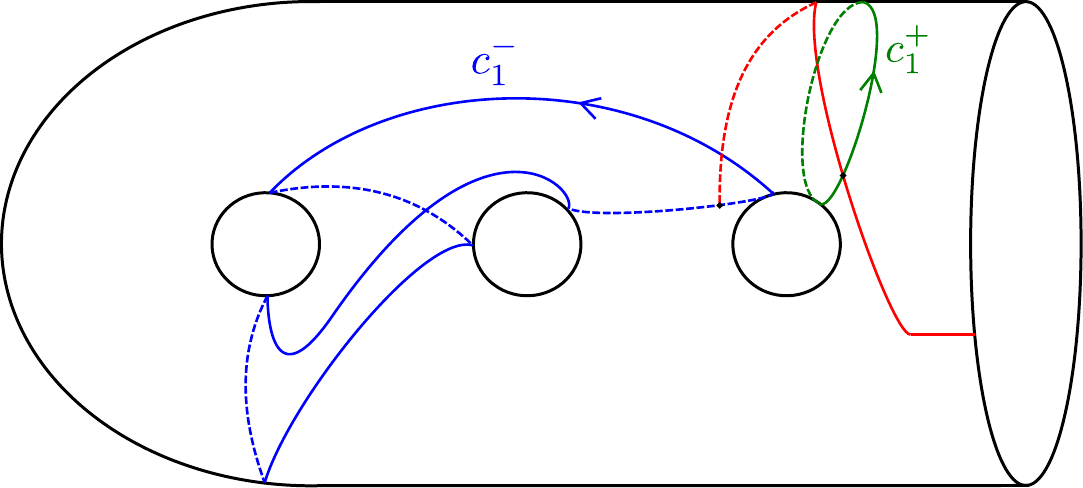} \qquad \qquad
  \includegraphics[scale=0.5]{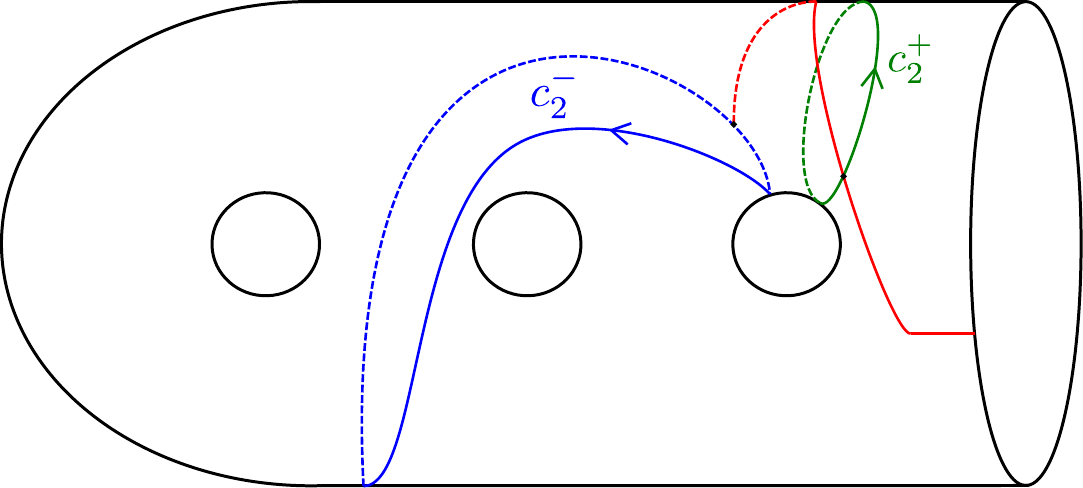}
\caption{The bounding pairs $(c_1^+, c_1^-)$ and $(c_2^+, c_2^-)$ in $\Sigma$} \label{fig:BP}
\end{figure}
 
Besides, let $\gamma_1,\gamma_2,\gamma_3,\gamma_4$ be the four   curves     shown in Figure \ref{fig:BSCC}:
note that each of $\gamma_1,\gamma_2$ bounds a   subsurface    of genus $2$,
and  each of $\gamma_3,\gamma_4$ bounds  a   subsurface     of genus $1$.
  Then consider the following product of  separating twists:   
$$
k:= T_{\gamma_1}\, T_{\gamma_2}^{-1}\, T_{\gamma_3}^{-1}\, T_{\gamma_4} \in \calK.
$$
 \begin{figure}[h]
 \includegraphics[scale=0.5]{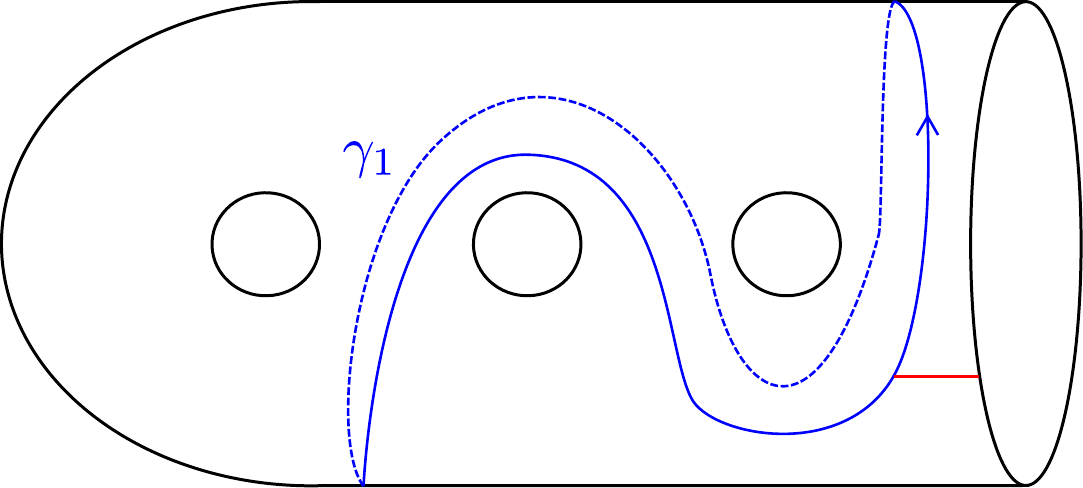} \qquad \qquad  \includegraphics[scale=0.5]{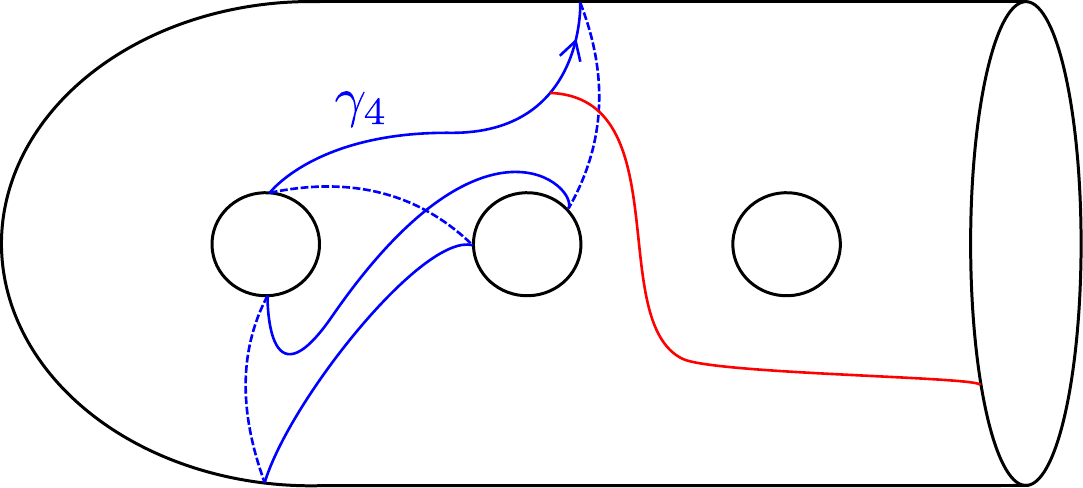}   \\[0.5cm]
  \includegraphics[scale=0.5]{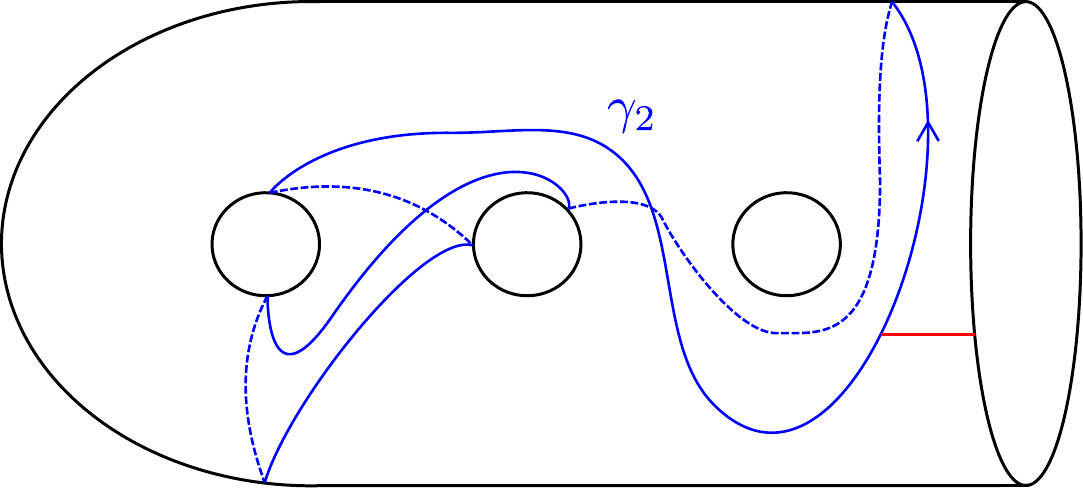}  \qquad \qquad \includegraphics[scale=0.5]{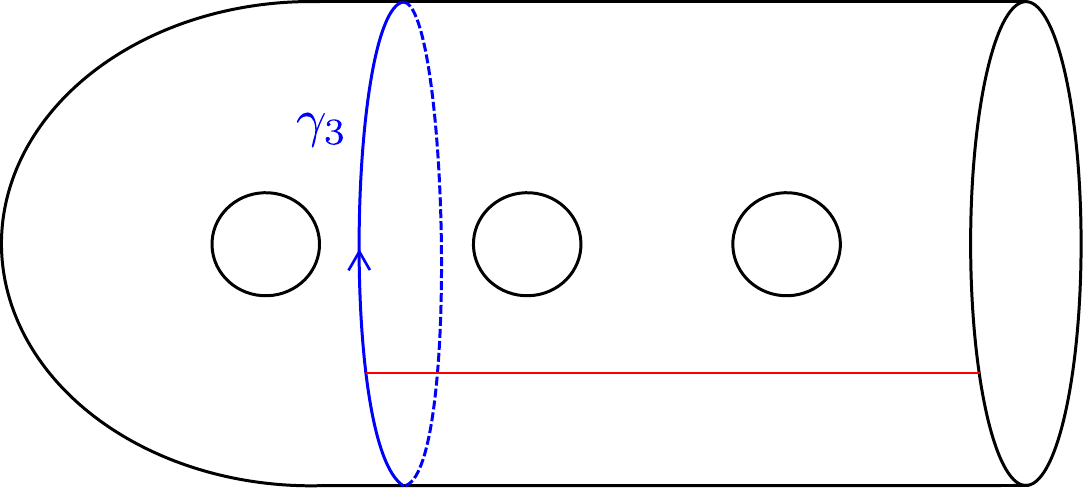} 
\caption{The separating  simple closed curves $\gamma_1, \gamma_2, \gamma_3$ and  $\gamma_4$ in $\Sigma$} \label{fig:BSCC}
\end{figure}

\begin{remark}
The above elements $i\in \calI$ and $k\in \calK$ can be alternatively described as products of ``commutators of simply intersecting pairs'',
which participate to Putman's infinite presentation of the Torelli group \cite{Put09}.
Recall that a \emph{commutator of simply intersecting pair} is a an element 
$[T_c,T_d]$ where $(c,d)$ is a pair of simple closed curves  meeting at two points such that $\omega(c,d)=0$.
Then it can be checked that
$$
i= \big[T_d,T_c^{-1}\big]
$$ 
where the curves  $c=c_2^-$ and $d$ are shown below:
$$
\includegraphics[scale=0.5]{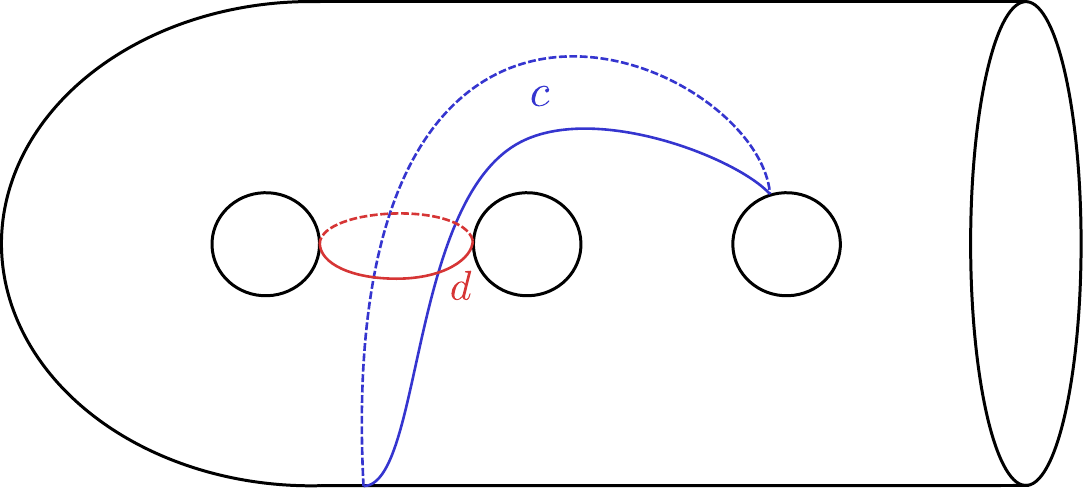}
$$
Similary, it can be verified that
$$
k= \big[T_e,T_d\big] \,\big[T_f^{-1}, T_d\big]
$$
where the curves  $e=\gamma_1$  and $f=\gamma_3$ are shown below:
$$
\includegraphics[scale=0.5]{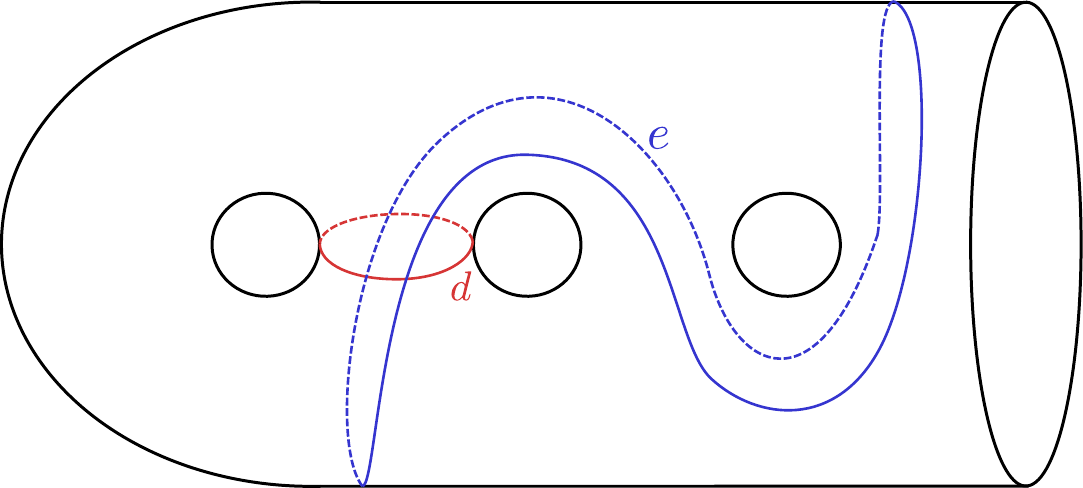} \qquad \includegraphics[scale=0.5]{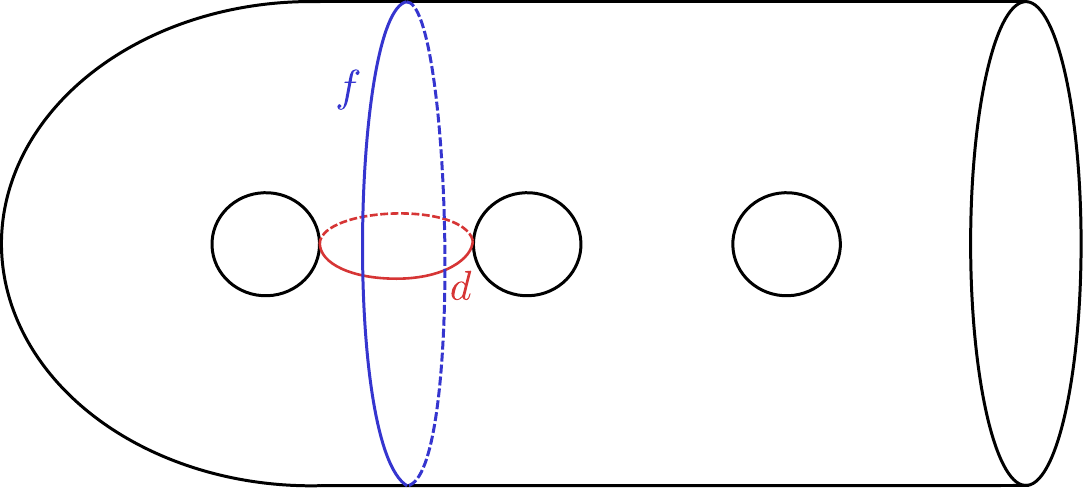}
$$ 

\up \hfill $\blacksquare$
\end{remark}

By construction, $\varphi = [i,k]$ belongs to $[\calI, \calK]$ and so to $\calM[3]$. 
Hence we can consider $\tau_3(\varphi)$,  the value of which can be deduced from the values of $\tau_1(i)$ and $\tau_2(k)$.
To compute $\tau_1(i)$, we use Johnson's formula \cite[Lemma 4.B]{Jo80a}   for a bounding pair map   
(or, alternatively, we can use \eqref{eq:r_1_BP} below):
$$
\tau_1\big(T_{c_1^{-}}^{-1}\,  T_{c_1^+} \big) =   -     a_3 \wedge a_1  \wedge (b_1 - a_2),
\qquad
\tau_1\big(T_{c_2^{-}}^{-1}\,  T_{c_2^+} \big) =  -      a_3 \wedge a_1 \wedge b_1
$$
Hence
\begin{equation}
\label{eq:tau1(i)}
\tau_1(i) =    -  a_3 \wedge a_1  \wedge (b_1 - a_2)    +  a_3 \wedge a_1 \wedge b_1 
=       a_1 \wedge a_2 \wedge a_3 =  \ltritree{a_2}{a_1}{a_3}.
\end{equation}

To compute now $\tau_2(k)$, we use Morita's formula \cite[Prop$.$ 1.1]{Mor89} for a   separating twist    
(or, alternatively, we can specialize formula \eqref{eq:r(DT)} below in degree $2$):
$$
\tau_2(T_{\gamma_1}) = \frac{1}{2}  (a_1 \wedge b_1 + a_3 \wedge b_3) \trait (a_1 \wedge b_1 + a_3 \wedge b_3), 
\quad  \tau_2(T_{\gamma_4}) = \frac{1}{2}  (a_1 \wedge (b_1-a_2)  ) \trait (a_1 \wedge (b_1-a_2) ) 
$$
$$
 \tau_2(T_{\gamma_2})  =  \frac{1}{2}  (a_1 \wedge (b_1-a_2)  + a_3 \wedge b_3) \trait (a_1 \wedge (b_1-a_2) + a_3 \wedge b_3),
\quad  \tau_2(T_{\gamma_3}) =  \frac{1}{2}  (a_1 \wedge b_1 ) \trait (a_1 \wedge b_1 )
 $$
Hence
\begin{equation} \label{eq:tau_2(k)}
\tau_2(k) = \frac{1}{2} \ltree{b_3}{a_3}{b_3}{a_3} + \ltree{b_3}{a_3}{b_1}{a_1} 
-  \frac{1}{2} \ltree{b_3}{a_3}{b_3}{a_3}   -   \ltree{b_3}{a_3}{b_1-a_2}{a_1}  = \ltree{b_3}{a_3}{a_2}{a_1}.
\end{equation}

We deduce that
\begin{equation} \label{eq:tau3}
\tau_3(\varphi) = \left[\tau_1(i), \tau_2(k)\right] = \left[ \ltritree{a_2}{a_1}{a_3}, \ltree{b_3}{a_3}{a_2}{a_1} \right]
=  \lfivetree{a_1}{a_2}{a_3}{a_2}{a_1} =0
\end{equation}
where the last identity follows from the AS relation. We conclude that $\varphi \in \calM[4]$.

\subsubsection{Beginning of the computation of $R(\varphi)$}

Since $\varphi$ belongs to $\calM[4]$, 
$R(\varphi)$ is simply the reduction of $\tau_4(\varphi)$ modulo $1$.
To compute this congruence class, we will use the infinitesimal Dehn--Nielsen representation $r^\theta$ described in \S \ref{subsec:infinitesimal}. We have
$$
 r^\theta(\varphi) =  r^\theta([i,k]) =  \big[r^\theta(i),r^\theta(k)\big]_\star
$$
where $[-,-]_\star$ denotes the commutator for the BCH product $\star$ on $\calT(H^\Q)$.
The latter is explicitly given by
$$
[u,v]_\star = u\star v \star (-u) \star (-v)  = [u,v] + \frac{1}{2} [u,[u,v]] -  \frac{1}{2} [v,[v,u]] + (\hbox{\small Lie brackets of length $\geq 4$} )
$$
for any $u,v\in \calT(H^\Q)$.
Since $u:=r^\theta(i)$ and $v:=r^\theta(k)$ start in degrees $1$ and $2$, respectively, we deduce that
\begin{eqnarray*}
r^\theta_4(\varphi) &=&  \big[r_1^\theta(i),r_3^\theta(k)\big] + \big[r_2^\theta(i),r_2^\theta(k)\big] 
+ \frac{1}{2}\big[r_1^\theta(i),\big [r_1^\theta(i),r_2^\theta(k)\big]\big] \\
&=&  \big[\tau_1(i),r_3^\theta(k)\big] + \big[r_2^\theta(i),\tau_2(k)\big] 
+ \frac{1}{2}\big[\tau_1(i),\big [\tau_1(i),\tau_2(k)\big]\big] 
\end{eqnarray*}
and, using \eqref{eq:tau3}, we obtain
\begin{equation} \label{eq:r4}
r^\theta_4(\varphi)  = \big[\tau_1(i),r_3^\theta(k)\big] + \big[r_2^\theta(i),\tau_2(k)\big].
\end{equation}
Hence, we are led to compute $r_2^\theta(i)$ and $r_3^\theta(k)$
and, for that, we will use the logansion given in Example \ref{ex:logansion}.
In fact, since we are only interested in the value of $r^\theta_4(\varphi)$ modulo $1$, 
we will only determine the classes of  $r_2^\theta(i)$ and $r_3^\theta(k)$ modulo   $\Z$-linear    combinations of trees.

\subsubsection{Computation of  $r_3^\theta(k)$}

The map $r^\theta_{[2,3]}$ is a group homomorphism on $\calK=\calM[2]$ (see \cite{Massuyeau}; this follows from Lemma \ref{lem:truncation} with $k:=2$). In particular, we have
$$
r_3^\theta(k) = r_3^\theta(T_{\gamma_1}) -  r_3^\theta(T_{\gamma_2}) - r_3^\theta(T_{\gamma_3})  + r_3^\theta(T_{\gamma_4}).
$$
Next, we shall compute   $r_3^\theta(T_{\gamma_i})$    for each $i\in\{1,2,3,4\}$ using the following formula,
which is deduced in \cite[eq. (5.4)]{KM} from the main result of \cite{KK}:

\begin{quote}
{\it For any   separating    simple closed curve $\gamma$ in $\Sigma$, and any representative $\tilde \gamma \in \pi$, we~have}
\begin{equation} \label{eq:r(DT)}
r^\theta(T_\gamma) = \frac 1 2 \theta(\tilde \gamma) \trait\theta(\tilde \gamma) \in \hat \calT(H^\Q).
\end{equation}
\end{quote}
We orient and base the curves $\gamma_i$ as shown in Figure \ref{fig:BSCC} to get the following lifts:
\begin{equation} \label{eq:gamma13}
\tilde \gamma_1 = [\alpha_3,\beta_3^{-1}] \beta_2 [\alpha_1,\beta_1^{-1}] \beta_2^{-1},  
 \qquad   \tilde \gamma_4 = [\alpha_2\beta_2\beta_1^{-1}, \alpha_1^{-1}][\alpha_1^{-1},\beta_2] ,
\end{equation}
\begin{equation} \label{eq:gamma24}
 \tilde \gamma_2 =  [\alpha_3,\beta_3^{-1}] [\alpha_2\beta_2\beta_1^{-1}, \alpha_1^{-1}]][\alpha_1^{-1},\beta_2], 
 \qquad  \tilde \gamma_3 = [\alpha_1,\beta_1^{-1}] .
\end{equation}
Then, a direct computation gives
\begin{eqnarray*}
\theta(\tilde \gamma_1) &=&-\treetwo{a_1}{b_1}-\treetwo{a_3}{b_3}   + \treethree{a_1}{b_1}{b_2}  + (\deg \geq 4)\\
\theta(\tilde \gamma_2) &=& -\treetwo{a_1}{b_1}+\treetwo{a_1}{a_2}-\treetwo{a_3}{b_3} 
 +\treethree{b_1}{a_1}{a_1}-\frac{1}{2}\treethree{a_2}{a_1}{a_1} -\frac{1}{2}\treethree{a_1}{a_2}{a_2} \\
&&  + \frac{1}{2}\treethree{a_1}{b_1}{a_2}  + \treethree{a_1}{b_1}{b_2} -\frac{1}{2}\treethree{b_2}{a_2}{a_1}  -\treethree{a_1}{b_2}{a_2} + (\deg \geq 4) \\
\theta(\tilde \gamma_3) &=& -\treetwo{a_1}{b_1} + (\deg \geq 4) \\
\theta(\tilde \gamma_4) &=&  -\treetwo{a_1}{b_1}+\treetwo{a_1}{a_2} + \treethree{b_1}{a_1}{a_1}-\frac{1}{2}\treethree{a_2}{a_1}{a_1} -\frac{1}{2}\treethree{a_1}{a_2}{a_2} \\
& & +\frac{1}{2}\treethree{a_1}{b_1}{a_2}  + \treethree{a_1}{b_1}{b_2}   -\frac{1}{2}\treethree{b_2}{a_2}{a_1}
-\treethree{a_1}{b_2}{a_2} + (\deg \geq 4)
\end{eqnarray*}
This computation can be performed either by hand or by using the SageMath code given in  Appendix \ref{sec:code}.
For instance, $\theta(\tilde \gamma_4)$ is computed and displayed with the following command lines:
\begin{lstlisting}[language=Python]
	g4=theta(comm('a2+b2+b1-','a1-')+comm('a1-','b2+'))
	display(g4)
\end{lstlisting}
We observe  that $r_3^\theta(T_{\gamma_3})$ is trivial, since $\theta_3(\tilde \gamma_3)=0$, 
and that  $r_3^\theta(T_{\gamma_1})$ is a   $\Z$-linear    combination of trees,
since  the above formula for $\theta_3(\tilde \gamma_1)=0$  shows no fraction.
Besides, we remark that $\theta_2(\tilde \gamma_2)= \theta_2(\tilde \gamma_4) - [a_3,b_3]$ and $\theta_3(\tilde \gamma_2)= \theta_3(\tilde \gamma_4)$;
therefore we have
\begin{eqnarray}
\notag r_3^\theta(k)  &\equiv & -  r_3^\theta(T_{\gamma_2})  + r_3^\theta(T_{\gamma_4}) \\
\notag &=& -  \theta_2(\tilde \gamma_2) \trait\theta_3(\tilde \gamma_2) 
+ \theta_2(\tilde \gamma_4) \trait \theta_3(\tilde \gamma_4) \\
\notag &=&  [a_3,b_3] \trait \theta_3(\tilde \gamma_2) \\
\label{eq:r_3(k)} & \equiv &  \frac{1}{2}  [a_3,b_3] \trait\big( [[a_2,a_1],a_1]  + [[a_1,a_2],a_2] + [[b_1,a_1],a_2]  + [[b_2,a_2],a_1]  \big)
\end{eqnarray}
where the symbol ``$\equiv$'' stands for a congruence modulo   $\Z$-linear    combinations of trees.

\subsubsection{Computation of  $r_2^\theta(i)$}

Lemma \ref{lem:truncation} with $k:=1$ implies that
\begin{equation} \label{eq:r_2}
r_2^\theta(i) = r_2^\theta(P_1) - r_2^\theta(P_2) - \frac{1}{2}[\tau_1(P_1), \tau_1(P_2)]
\quad \hbox{where } P_i := T_{c_i^{-}}^{-1}\,  T_{c_i^+}.
\end{equation}
We  need the following formulas for   bounding pair maps.   

\begin{proposition} \label{prop:BP}
Let $\gamma$ and $\delta$ be elements of $\pi$ representing two simple closed curves   that cobound a subsurface of $\Sigma$,   
and set $c := \gamma^{-1} \delta$. Then  we have
\begin{equation} \label{eq:r_1_BP}
r_1^\theta(T_\gamma T_\delta^{-1}) = -  {[\gamma]} \trait{[c]} 
\end{equation} 
and
\begin{equation} \label{eq:r_2_BP}
r_2^\theta(T_\gamma T_\delta^{-1}) = - \frac{1}{2} {[c]}\trait{[c]}  - {\theta_2(\gamma)}\trait{[c]}  - {[\gamma]}\trait{\theta_3(c)}  
\end{equation} 
where 
$$
[\gamma]\in \frac{\pi}{\Gamma_2\pi} \simeq H
\quad \hbox{and} \quad [c] \in \frac{\Gamma_2 \pi }{\Gamma_3 \pi} \simeq \frakL_2
$$ 
are the leading terms of 
$\theta(\gamma) =[\gamma] + \theta_2(\gamma) +\!\cdots \in \hat\frakL^\Q$
 and $\theta(c) =[c] + \theta_3(c) +\!\cdots \in \hat\frakL^\Q$, respectively.
\end{proposition}

\begin{proof}
There is also a version of \eqref{eq:r(DT)} for the Dehn twist $T_\gamma$ along \emph{any} simple closed curve $\gamma$.
Of course, $T_\gamma$ does not belong to the Torelli group if $\gamma$ is not   separating   , but the automorphism 
$\varrho^\theta(T_\gamma)$ of $\hat\frakL^\Q$ still has a logarithm and we can consider
$$
r^\theta(T_\gamma):= (\eta^\Q)^{-1}\big( \log \varrho^\theta(T_\gamma) \big) \  \in \hat\calT(H^\Q)
$$
by allowing tree diagrams to have no trivalent vertices
and derivations to be of degree $0$.
(Hence the degree $0$ part of $\calT(H^\Q)$ is canonically isomorphic to $S^2(H^\Q)$.)
This generalization of formula~\eqref{eq:r(DT)} is proved exactly as in \cite[eq. (5.4)]{KM} using the main result of \cite{KK}.

Since $T_\gamma$ and $T_\delta$ commute, 
the automorphisms $\varrho^\theta(T_\gamma)$ and $\varrho^\theta(T_\delta)$ commute,
and so do their logarithms.
Hence we have $[r^\theta(T_\gamma), r^\theta(T_\delta)]=0$, so that the BCH formula reduces to
$
r^\theta(T_\gamma T_\delta^{-1}) = r^\theta(T_\gamma) + r^\theta(T_\delta^{-1}),
$ 
and we deduce that
$$
r^\theta(T_\gamma T_\delta^{-1}) = 
 \frac 1 2 \theta( \gamma) \trait\theta( \gamma) -  \frac 1 2 \theta(\delta) \trait\theta(\delta).
$$
Besides, we have
$$
\theta(\delta) = \theta(\gamma) \star \theta(c) =  \theta(\gamma) +  \theta(c) + \frac 1 2 \big[ \theta(\gamma), \theta(c) \big]+  (\deg \geq 4)
$$
since $\theta(c)$ starts in degree $2$. Thus we obtain 
$$
r^\theta(T_\gamma T_\delta^{-1}) =   - \frac{1}{2} \theta(c) \trait \theta(c)  - \theta(\gamma) \trait \theta(c)  + (\deg\geq 3)
$$
since we have $ \theta(\gamma) \trait \big[ \theta(\gamma), \theta(c) \big] =0$ by the AS relation.
In particular, we get
$$
r^\theta_1 (T_\gamma T_\delta^{-1}) = - \theta_1(\gamma) \trait \theta_2(c)
$$
and 
$$
r^\theta_2 (T_\gamma T_\delta^{-1}) =    - \frac{1}{2} \theta_2(c) \trait \theta_2(c)   -  \theta_2(\gamma) \trait \theta_2(c) -  \theta_1(\gamma) \trait \theta_3(c) .
$$

\up
\end{proof}

We now apply \eqref{eq:r_2_BP}  to compute $r_2^\theta(P_1)$ and  $r_2^\theta(P_2)$ modulo   $\Z$-linear    combinations of trees:
\begin{enumerate}
\item For $P_1$, we consider $\gamma:=\alpha_3$ and 
$c:=  \beta_3 \tilde \gamma_4^{-1} \beta_3^{-1}$ where  $\tilde \gamma_4$ has been defined at \eqref{eq:gamma24}. 
Then, a direct computation (using the SageMath code of Appendix \ref{sec:code}) gives
\begin{eqnarray*}
\theta(\gamma) &=& \stick{a_3}  -\frac{1}{2} \treetwo{a_3}{b_3} - \frac{1}{2} \treethree{a_1}{b_1}{a_3} 
-\frac{1}{2}  \treethree{a_2}{b_2}{a_3} + \frac{1}{12}  \treethree{a_3}{b_3}{b_3} +(\deg \geq 4),\\
\theta(c)  &=& \treetwo{a_1}{b_1} - \treetwo{a_1}{a_2}  -\frac{1}{2}\treethree{a_1}{a_2}{a_1}+\treethree{a_1}{b_1}{a_1}
+\frac{1}{2}\treethree{a_1}{a_2}{a_2}+\frac{1}{2}\treethree{a_2}{b_2}{a_1}   -\frac{1}{2}\treethree{a_1}{b_1}{a_2}\\
&& - \treethree{a_1}{b_1}{x} + \treethree{a_1}{a_2}{x}  +(\deg \geq 4) \quad \hbox{ with } x:=b_2+b_3. 
\end{eqnarray*}
Hence we get 
\begin{eqnarray*}
  r_2^\theta(P_1) & = &  
 \Bigg(-\frac{1}{2}\ltree{b_1}{a_1}{b_1}{a_1}-\frac{1}{2} \ltree{a_2}{a_1}{a_2}{a_1} + \ltree{a_2}{a_1}{b_1}{a_1}\Bigg) 
+ \Bigg(\frac{1}{2} \ltree{b_3}{a_3}{b_1}{a_1} - \frac{1}{2} \ltree{b_3}{a_3}{a_2}{a_1}\Bigg)  \\
&&  + \Bigg( \frac{1}{2} \ltree{a_2}{a_1}{a_3}{a_1}- \ltree{b_1}{a_1}{a_3}{a_1} -\frac{1}{2}\ltree{a_2}{a_1}{a_3}{a_2} -\frac{1}{2} \ltree{b_2}{a_2}{a_3}{a_1}  \\
&&  \quad +\frac{1}{2} \ltree{b_1}{a_1}{a_3}{a_2} +\ltree{b_1}{a_1}{a_3}{x}-\ltree{a_2}{a_1}{a_3}{x}\Bigg) \\
&\equiv& \frac{1}{2} \Bigg( \ltree{a_1}{a_2}{a_3}{a_1+a_2}   + \ltree{a_1}{b_1}{a_3}{a_2}  +  \ltree{a_2}{b_2}{a_3}{a_1} \\
&&+  \ltree{b_3}{a_3}{b_1+a_2}{a_1} + \ltree{a_1}{b_1}{a_1}{b_1} +  \ltree{a_1}{a_2}{a_1}{a_2} \Bigg)
\end{eqnarray*}
where the symbol ``$\equiv$'' stands for a congruence modulo   $\Z$-linear    combinations of trees.
\item For $P_2$, we consider now $\gamma:=\alpha_3$ and 
$c:= (\beta_3\beta_2) \tilde\gamma_3^{-1} (\beta_3\beta_2)^{-1}$ where  $\tilde \gamma_3$ has been defined at \eqref{eq:gamma13}. 
Then, a direct computation (using the SageMath code of Appendix \ref{sec:code}) gives
\begin{eqnarray*}
\theta(\gamma) &=& \stick{a_3}  -\frac{1}{2} \treetwo{a_3}{b_3} - \frac{1}{2} \treethree{a_1}{b_1}{a_3} 
-\frac{1}{2}  \treethree{a_2}{b_2}{a_3} + \frac{1}{12}  \treethree{a_3}{b_3}{b_3} +(\deg \geq 4),\\
\theta(c)  &=& \treetwo{a_1}{b_1} - \treethree{a_1}{b_1}{x}  +(\deg \geq 4) \quad \hbox{ with } x:=b_2+b_3.
\end{eqnarray*}
Hence we get 
\begin{eqnarray*}
r_2^\theta(P_2) &=&  -\frac{1}{2} \ltree{b_1}{a_1}{b_1}{a_1}  + \frac{1}{2} \ltree{b_3}{a_3}{b_1}{a_1} +  \ltree{b_1}{a_1}{a_3}{x}   \\
&\equiv&   \frac{1}{2} \Bigg( \ltree{b_3}{a_3}{b_1}{a_1}  +  \ltree{b_1}{a_1}{b_1}{a_1}  \Bigg).
\end{eqnarray*}
\end{enumerate}

We now insert into \eqref{eq:r_2} the above values of $r_2^\theta(P_i)$ to get 
\begin{eqnarray}
\notag r_2^\theta(i) &\equiv&  \frac{1}{2} \Bigg( \ltree{a_1}{a_2}{a_3}{a_1+a_2}   + \ltree{a_1}{b_1}{a_3}{a_2}  +  \ltree{a_2}{b_2}{a_3}{a_1} 
+  \ltree{b_3}{a_3}{a_2}{a_1}  +  \ltree{a_1}{a_2}{a_1}{a_2} \Bigg)\\
\notag && - \frac{1}{2} \Bigg[\ltritree{a_3}{a_1}{b_1} - \ltritree{a_3}{a_1}{a_2}, \ltritree{a_3}{a_1}{b_1}\Bigg] \\
\label{eq:r2i} &\equiv & \frac{1}{2} \Bigg( \ltree{a_1}{a_2}{a_3}{a_1+a_2}   + \ltree{a_1}{b_1}{a_3}{a_2}  +  \ltree{a_2}{b_2+a_3}{a_3}{a_1} 
+  \ltree{b_3}{a_3}{a_2}{a_1}  +  \ltree{a_1}{a_2}{a_1}{a_2} \Bigg)
\end{eqnarray}

\subsubsection{End of the computation of $R(\varphi)$}

It follows from \eqref{eq:tau1(i)} and \eqref{eq:r_3(k)} that
\begin{eqnarray*}
\big[\tau_1(i),r_3^\theta(k)\big]  & \equiv &  \frac{1}{2} \Bigg[ \ltritree{a_2}{a_1}{a_3},  
\lfivetree{a_3}{b_3}{a_1}{a_1}{a_2}+ \lfivetree{a_3}{b_3}{a_2}{a_2}{a_1} + \lfivetree{a_3}{b_3}{a_2}{a_1}{b_1} + \lfivetree{a_3}{b_3}{a_1}{a_2}{b_2}\Bigg]\\
& \equiv & \frac{1}{2} \Bigg( \lsixtree{a_1}{a_2}{a_3}{a_1}{a_1}{a_2} +  \lsixtree{a_1}{a_2}{a_3}{a_2}{a_2}{a_1}
+  \lsixtree{a_1}{a_2}{a_3}{a_2}{a_1}{b_1} +  \lsixtree{a_1}{a_2}{a_3}{a_1}{a_2}{b_2}\\
&&+  \lsixtree{a_3}{b_3}{a_2}{a_1}{a_2}{a_3} + \lsixtree{a_3}{b_3}{a_1}{a_2}{a_3}{a_1}  \Bigg).
\end{eqnarray*}
Besides, it follows from \eqref{eq:tau_2(k)} and \eqref{eq:r2i} that
\begin{eqnarray*}
\big[r_2^\theta(i), \tau_2(k)\big]  &\equiv&  \frac{1}{2}  \Bigg[  \ltree{a_1}{a_2}{a_3}{a_1+a_2}   + \ltree{a_1}{b_1}{a_3}{a_2}  +  \ltree{a_2}{b_2+a_3}{a_3}{a_1} 
   , \ltree{b_3}{a_3}{a_2}{a_1}\Bigg]  \\
   & \equiv & \frac{1}{2} \Bigg( \lsixtree{a_2}{a_1}{a_1}{a_3}{a_1}{a_2} + \lsixtree{a_2}{a_1}{a_2}{a_3}{a_1}{a_2} 
   + \lsixtree{b_1}{a_1}{a_2}{a_3}{a_1}{a_2} + \lsixtree{a_3}{b_3}{a_2}{a_1}{a_2}{a_3} \\
   &&   + \lsixtree{b_2}{a_2}{a_1}{a_3}{a_1}{a_2}  + \lsixtree{a_3}{a_2}{a_1}{a_3}{a_1}{a_2} 
   + \lsixtree{a_1}{a_2}{a_3}{a_2}{a_1}{a_3}  +   \lsixtree{b_3}{a_3}{a_1}{a_2}{a_1}{a_3} \Bigg) .
\end{eqnarray*}
We deduce from \eqref{eq:r4} that 
$$
r_4^\theta(\varphi) \equiv
 \frac{1}{2} \Bigg(   \lsixtree{a_3}{a_2}{a_1}{a_3}{a_1}{a_2}     + \lsixtree{a_1}{a_2}{a_3}{a_2}{a_1}{a_3}   \Bigg) 
\equiv   \frac{1}{2} \lsixtree{a_1}{a_2}{a_3}{a_3}{a_1}{a_2}
$$
where the last congruence follows from the IHX relation. Thus we have proved \eqref{eq:dos}.

\subsection{Another proof of Theorem B} \label{subsec:another_proof_B}

Our second proof of Theorem B  is based  on $3$-dimensional topology and,   to be more specific, on the clasper calculus in homology cylinders.   
It is   inspired by    the arguments of Nozaki, Sato and Suzuki \cite{NSS} to prove Theorem A in the closed case.
(See \cite[\S 5.4]{NSS} in connection to this.)
But,   in contrast to \cite{NSS}, our arguments do not involve any computation of the LMO homomorphism.   

To be more specific, still assuming that the surface $\Sigma$ has genus $g\geq 3$,
we will show here the existence of an element $\varphi' \in \Gamma_3 \calI$ of the form
$$
\varphi'=[\varphi_1,[\varphi_2,\varphi_3]]
$$
where $\varphi_1,\varphi_2,\varphi_3\in \calI$ will be required to satisfy certain properties. It will follow from these properties that $\varphi'$ belongs
to ${\calM}[4]$ and  satisfies
\begin{eqnarray}
\label{eq:dos'} R(\varphi')&=& \frac{1}{2} \lsixtree{a_3}{a_2}{a_1}{a_1}{a_2}{a_3} \mod 1.
\end{eqnarray} 
Then, Theorem B is proved with $\varphi'$ exactly as we did in \S \ref{subsec:proof_B} with the first element $\varphi$.

Let $\calC := \calC(\Sigma)$ be the monoid of \emph{homology cobordisms} over $\Sigma$:
the reader is refered to \cite{HM} for a survey.
Recall that  the mapping class group  embeds into this monoid via the \emph{mapping cylinder} construction
$$
\mathbf{c} : \calM \longrightarrow \calC
$$
and that most of the constructions outlined in \S \ref{subsec:infinitesimal}, \S \ref{subsec:truncations} for $\calM$ can be extended to $\calC$.
Thus, we have at our disposal the Johnson filtration 
\begin{equation} \label{eq:Johnson_C}
\calC \supset \underbrace{\calC[1]}_{\calIC\, :=} \supset \underbrace{\calC[2]}_{ \calKC\, :=} \supset \cdots \supset \calC[k] \supset \calC[k+1] \supset \cdots
\end{equation}
and, for every $k\geq 1$, the $k$-th Johnson homomorphism
$
\tau_k: \calC[k] \to \mathsf{D}_k(H)
$
which, by work of Garoufalidis and Levine \cite{GL}, is  surjective.
The submonoid $\calC[1] = \calIC$ consists of \emph{homology cylinders} over $\Sigma$, 
and the infinitesimal Dehn--Nielsen representation $r^\theta: \calIC \to \hat\calT(H^\Q)$ is defined on this monoid.
Besides, the map  $R_\circ : \calK \to \calT_4(H^\Q)/ \calT_4(H)$ that has been defined in Remark~\ref{rem:R_0} as a variation of $R$
extends to $\calKC$ by the same formula \eqref{eq:R_0},
and the same arguments   show that the resulting map
$$
R_\circ: \calKC \longrightarrow \frac{\calT_4(H^\Q)}{\calT_4(H)} 
$$
is a monoid homomorphism. 

To go further, we shall need the $Y_k$-equivalence relations on $\calC$ 
that have been introduced by Goussarov \cite{Goussarov} and Habiro  \cite{Habiro}. 
Recall that two cobordisms $M,M'\in \calC$ are  \emph{$Y_k$-equivalent} if there is a (compact, connected, oriented) surface $S \subset \hbox{int}(M)$
with one boundary component, and an element $s\in \Gamma_k \calI(S)$, such that $M'$ is obtained by cutting open $M$ along $S$
and gluing it back with $s$. We need the following facts about these equivalence relations 
(see the survey paper \cite{HM}, and references therein):

\begin{itemize}
\item For every $k\geq 1$, the $Y_k$-equivalence relation is generated by surgeries along connected graph claspers 
with $k$ nodes (using the terminology of \cite{Habiro}).
\item Denoting by $Y_k \calIC$ the $Y_k$-equivalence class of the trivial cylinder $U:=\Sigma \times [-1,+1]$,
we obtain a decreasing sequence of submonoids
$$
\calIC = Y_1 \calIC \supset Y_2 \calIC \supset \cdots \supset Y_k \calIC \supset Y_{k+1} \calIC \supset \cdots
$$
which is called the \emph{$Y$-filtration} and is smaller than the Johnson filtration \eqref{eq:Johnson_C}.
\item For every $k\geq 1$, the quotient monoid $\calIC/Y_k$ is a group and, for all $\ell,\ell'\geq 1$, we have
$$
\Big[ \frac{Y_\ell\calIC}{Y_k} ,  \frac{Y_{\ell'}\calIC}{Y_k}\Big] \subset \frac{Y_{\ell + \ell'} \calIC}{Y_k}.
$$
\item The associated graded of the $Y$-filtration, i.e$.$ the direct sum of abelian groups
$$
\operatorname{Gr}^Y \calIC := \bigoplus_{k=1}^{+\infty}  \frac{Y_{k} \calIC}{Y_{k+1}},
$$
is a graded Lie ring whose Lie bracket is induced by group commutators.                                       
\item The graded Lie ring $\operatorname{Gr}^Y \calIC$ can be ``approximated'' by a space of \emph{Jacobi diagrams} in the following way.
A \emph{Jacobi diagram} is a finite and  unitrivalent graph, whose trivalent vertices are oriented
and whose univalent vertices are colored by the finite set
$$
\{1^+,\dots,g^+\} \cup \{1^-,\dots,g^-\}.
$$ 
The \emph{degree} of a Jacobi diagram is the number of its trivalent vertices.
Let $\calA^Y$ be the graded abelian group generated by Jacobi diagrams, 
subject to the AS and IHX relations as presented in \S \ref{subsec:trees}.
Equipped with the multiplication $\star$ defined by
$$
D \star E := \sum \Big( \begin{array}{c} \hbox{\small all possible ways of gluing some of the  $i^+$-vertices of $D$}\\
 \hbox{\small with some of the  $i^-$-vertices of $E$,  for all $i\in\{1,\dots,g\}$} \end{array}\Big),
$$
the graded abelian group  $\calA^Y$ is a graded ring. Furthermore,
equipped with the bracket $[D,E]_\star := D \star E-E \star D$,
the subgroup $\calA^{Y,c}$ of $\calA^Y$ spanned by connected Jacobi diagrams is a graded Lie ring. 
Then, there is a homomorphism of graded Lie rings \cite{GL,CHM}
\begin{equation} \label{eq:psi}
\psi: \calA^{Y,c} \longrightarrow \operatorname{Gr}^Y \calIC 
\end{equation}
defined by $\psi(D) := (-1)^{\chi(D)} \cdot \big( U_{\overline D}\!  \mod Y_{k+1}\big)$ for any connected Jacobi diagram $D$ of degree $k$,
where $\chi(D)$ is the Euler characteristic of $D$ and $\overline D$ is a graph clasper in the trivial cylinder $U$ ``realizing'' $D$:
in particular, every univalent vertex of $D$ is ``realized'' by a leaf of $\overline D$
which is a push-off (in the interior of $U$) of the framed curve $\alpha_i \subset \Sigma \times \{-1\}$ 
(resp. $\beta_i  \subset \Sigma \times \{+1\}$) if the color of that vertex is $i^-$ (resp. $i^+$).
\end{itemize}

We can now prove the following.

\begin{lemma} \label{lem:Y4}
The monoid homomorphism $R_\circ:  \calKC \to \calT_4(H^\Q)/ \calT_4(H)$ factorizes to
a group homomorphism $R_\circ:  \calKC/Y_4 \to \calT_4(H^\Q)/ \calT_4(H)$.
\end{lemma}

\begin{proof}
Let $M \in \calKC$ and let $G\subset \hbox{int}(M)$ be a connected graph clasper with $4$ nodes.
Since the $Y_4$-equivalence is generated by surgeries of the type $M \leadsto M_G$, it is enough to prove that 
\begin{equation}\label{eq:rrrr}
R_\circ(M_G)=R_\circ(M).
\end{equation}
Let $N_+$ be a collar neighborhood  of the ``top'' boundary of the cobordism $M$.
Since $M$ is a homology cobordism, each leaf of $G$ is homologous to a framed knot contained in $N_+$.
Hence, by standard techniques of clasper calculus, we can find another connected graph clasper $G' \subset N_+$  with $4$ nodes such that 
\begin{equation} \label{eq:Y5}
M_G \sim_{Y_5} M_{G'}.
\end{equation}
Since the map $R_\circ$ is determined by $r_{[2,4]}^\theta$, it is determined by the action of $\calKC$ on 
the $5$-th nilpotent quotient $\pi/\Gamma_6 \pi$ of $\pi$. Hence $R_\circ$ factorizes through the $Y_5$-equivalence, 
and we deduce from \eqref{eq:Y5} that
$$
R_\circ(M_G) =R_\circ(M_{G'}).
$$
Identify $N_+$ with the trivial cylinder $U$ using the ``top'' boundary parametrization of $M$:
then $G' \subset N_+$ corresponds to yet another graph clasper $G'' \subset U$ with 4 nodes. Thus we obtain
$$
R_\circ(M_G) = R_\circ(M \circ U_{G''}) = R_\circ(M) + R_\circ(U_{G''}).
$$ 
Besides, $R_\circ(U_{G''}) = (\tau_4(U_{G''})\!\! \mod 1) $ is trivial, 
because $\tau_4(U_{G''})\in \calT_4(H^\Q)$
is $0$ if the Jacobi diagram $\underline{G''}$ underlying $G''$ is looped
and is  equal to $\underline{G''} \in  \calT_4(H)$ otherwise \cite{GL}. 
Thus, we conclude to~\eqref{eq:rrrr}.
\end{proof} 

Consider the following Jacobi diagrams of degree $1$:
$$
T_1 := \ltritree{1^-}{2^-}{3^-}, \ T_2 := \ltritree{1^-}{1^+}{2^-}, \ T_3 := \ltritree{2^-}{2^+}{3^-}  \in \calA^Y_1,
$$
and note that 
$$
\big[T_1, \big[T_2,T_3\big]_\star\big]_\star  = \underbrace{\lfivetree{3^-}{2^-}{1^-}{2^-}{3^-}}_{T:=} \in \calA^Y_3.
$$
Consider now  a graph clasper in the trivial cylinder $U= \Sigma \times [-1,+1]$ of the following form:
$$
 \centre{\labellist \small \hair 2pt
 \pinlabel {$\overline T:=$} [r] at -80 200
\pinlabel {$\alpha''_3$} [bl] at 30 250
\pinlabel {$\alpha''_2$} [bl]  at 30 30
\pinlabel {$\alpha'_3$} at 0 220
\pinlabel {$\alpha'_2$} at 0 0
\pinlabel {$\alpha_1$} at 600 100
\endlabellist}
\includegraphics[scale=0.2]{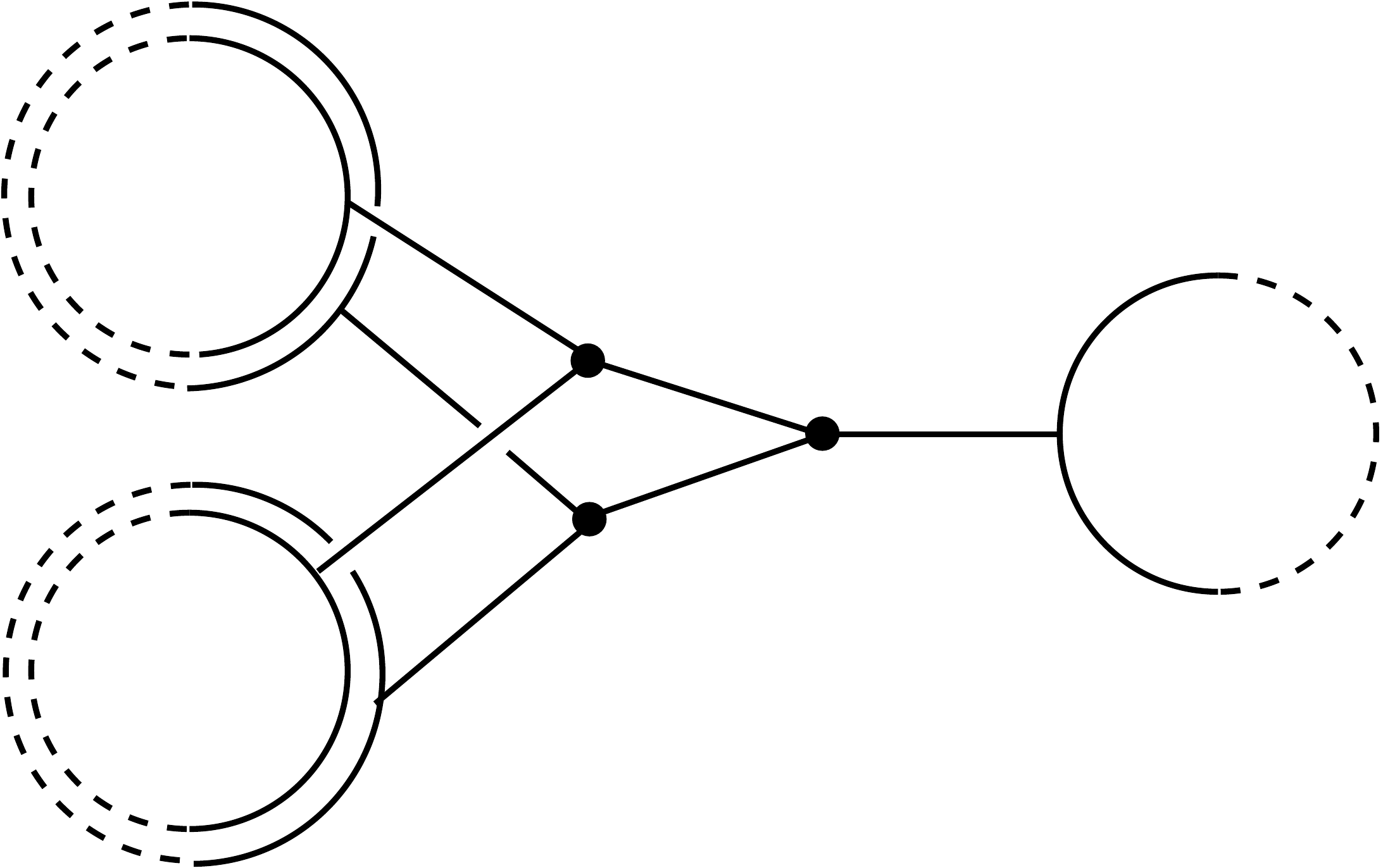}
$$
where $\alpha_1$ denotes a push-off of the framed curve $\alpha_1 \subset \Sigma \times \{-1\}$ and, for $i\in\{2,3\}$,
$\alpha'_i,\alpha''_i$ denote parallel copies of a push-off of the framed curve $\alpha_i \subset \Sigma \times \{-1\}$; note that
$$
\tau_3(U_{\overline T}) = T\vert_{i^-\mapsto \alpha_i} =0  \quad \hbox{(by the AS relation)}
$$
so that $U_{\overline T}$ belongs to $\calC[4]$.
Consider also a graph clasper in $U$ of the following form:
$$
{\labellist \small \hair 2pt
 \pinlabel {$S:=$} [r] at -80 200
\pinlabel {$\alpha_3$} at 0 220
\pinlabel {$\alpha_2$} at 0 0
\pinlabel {$\alpha_1$} at 600 100
\endlabellist}
\includegraphics[scale=0.2]{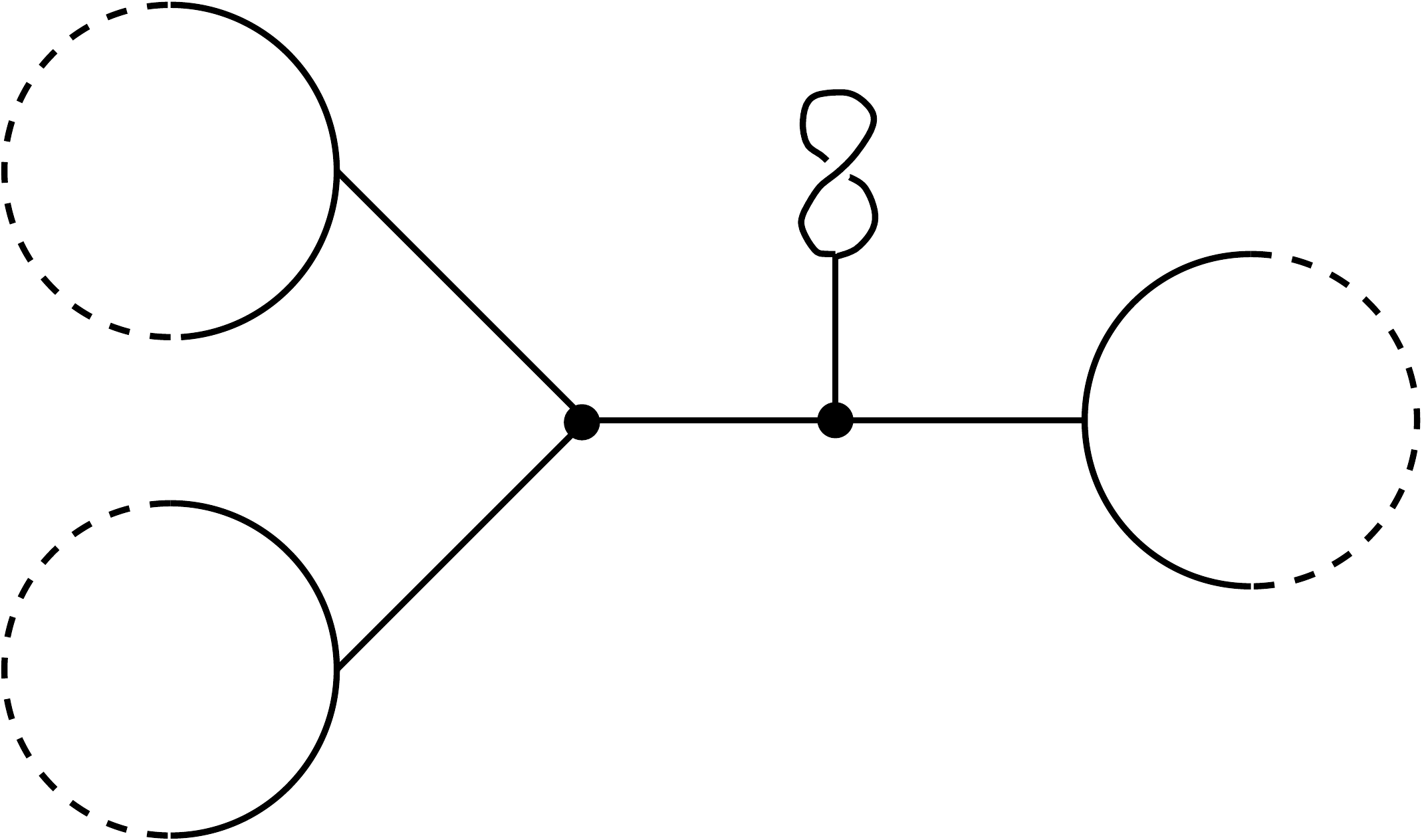}
$$
 By a result of Conant, Schneiderman and Teichner \cite[proof of Lemma 33]{CST16}, the homology cylinders $U_{\overline T} $ and $U_{S}$ are, 
 up to surgeries along graph claspers with $4$ nodes, related by a $4$-dimensional 
 homology cobordism. Since $R_\circ$ is invariant under $4$-dimensional homology cobordism (because $r^\theta$ is so)
 and since $R_\circ$ is invariant under the $Y_4$-equivalence too (by Lemma \ref{lem:Y4}), we deduce that
 \begin{equation} \label{eq:RRR}
R_\circ(U_{\overline T}) = R_\circ(U_{S}).
 \end{equation}
 It follows also from \cite[\S 3.8]{CST16} (or, alternatively, from  \cite[Th$.$ B]{KM} which is more general) that
 \begin{equation} \label{eq:tau_4(U_T)}
 \tau_4(U_{S}) = \frac{1}{2} \lsixtree{a_3}{a_2}{a_1}{a_1}{a_2}{a_3}.
 \end{equation}
Besides, since the map \eqref{eq:psi} is a  Lie homomorphism and $\overline{T}$ is a ``realization'' of ${T}$, we have 
 $$
 \big[\psi_1(T_1), \big[\psi_1(T_2),\psi_1(T_3)\big]\big] =  \psi_3(T) =  - \big( U_{\overline T} \mod Y_4 \big).
 $$
 Next, since $\mathbf{c}: \calI / \Gamma_2 \calI \to \calIC/ Y_2\calIC$ is an isomorphism  in genus $g\geq 3$ \cite{MM}, 
 we can find $\varphi_i\in \calI$ such that
 \begin{equation} \label{eq:cond_2}
\psi_1(T_i) = \big(\mathbf{c}(\varphi_i) \!\! \mod Y_2\big).
 \end{equation}
So, the  mapping cylinder of the inverse of $\varphi':=[\varphi_1,[\varphi_2,\varphi_3]] \in \Gamma_3\!  \calI$ is $Y_4$-equivalent to $ U_{\overline T} $:
in particular, we have $\tau_3(\varphi')=-\tau_3(U_{\overline T})=0$ so that $\varphi' \in \calM[4]$.
 Finally, thanks to Lemma \ref{lem:Y4}, we conclude  from \eqref{eq:RRR} and \eqref{eq:tau_4(U_T)} that
 $$
 R_\circ(\varphi')  = R_\circ(U_{\overline T})  = \frac{1}{2} \lsixtree{a_3}{a_2}{a_1}{a_1}{a_2}{a_3}\! \mod 1.
 $$

\begin{remark}
We can give an explicit example of an element $\varphi_i \in \calI$ satisfying \eqref{eq:cond_2} for $i\in \{1,2,3\}$,
which leads to an explicit formula for $\varphi'=[\varphi_1,[\varphi_2,\varphi_3]]$. Indeed, according to \cite[Th$.$ 1.3]{MM},
the property \eqref{eq:cond_2}  is equivalent to the double condition
\begin{equation} \label{eq:cond_2_bis}
 \tau_1(\varphi_i) = \tau_1(\psi(T_i)) \in \Lambda^3 H \quad \hbox{and}  \quad  \beta(\varphi_i) = \beta(\psi(T_i)) \in B_{\leq 3},
\end{equation}
where $\beta$ denotes the Birman--Craggs homomorphism with values in the space $B_{\leq 3}$ of cubic  boolean functions on the space
$\hbox{Spin}(\Sigma)$ of spin structures on $\Sigma$. By \cite[Lemma 4.22]{MM} and \cite[Lemma 4.23]{MM}, respectively, we have
$$
 \tau_1(\psi(T_1))  =  \ltritree{a_2}{a_1}{a_3}, \quad
 \tau_1(\psi(  T_2 )) = \ltritree{b_1}{a_1}{a_2}, \quad
 \tau_1(\psi(  T_3 )) = \ltritree{b_2} {a_2}{a_3}
$$ 
and 
$$
 \beta(\psi(T_1))  =  \overline{a_2} \cdot \overline{a_1} \cdot \overline{a_3}, \quad
 \beta(\psi(  T_2 )) = \overline{b_1} \cdot \overline{a_1}  \cdot \overline{a_2}, \quad
 \beta(\psi(  T_3 )) = \overline{b_2}  \cdot  \overline{a_2}  \cdot \overline{a_3}.
$$ 
Here, identifying $\hbox{Spin}(\Sigma)$ with the space $\mathcal{Q}$ of quadratic functions $H \otimes \mathbb{Z}_2  \to \Z_2$
whose polar form is the mod $2$ intersection form of the surface, we associate to any $z\in H$
the affine boolean  map  $\overline{z}:\mathcal{Q} \to  \Z_2$  defined by $\overline{z}(q):= q(z\otimes 1)$.
Thus, using \cite[Lemma 4.B]{Jo80a} and \cite[\S 7]{Jo80b},  
we see that the following instances of $\varphi_1,\varphi_2,\varphi_3$ satisfy \eqref{eq:cond_2_bis}:
$$
\varphi_1 := i \,(T_{d'}T_{d''}), \quad  \varphi_2 := \big(T_{e^+} T_{e^-}^{-1}\big) \, T_{f} , \quad  \varphi_3 := \big(T_{u^+} T_{u^-}^{-1}\big) \, T_{v}
$$
  Here $i$ is the product of two bounding pair maps defined by \eqref{eq:i}
 and $T_{d'},T_{d''}$ are  the separating twists along the curves   $d',d''$ given by Figure \ref{fig:more_curves}, thus defining $\varphi_1$;
 also, $\varphi_2$ and $\varphi_3$ are defined as products of a bounding pair map and a separating twist, whose
 curves  are also shown  in Figure~\ref{fig:more_curves}.   \hfill $\blacksquare$
  \begin{figure}[h!]
 \begin{tabular}{c}
 \includegraphics[scale=0.5]{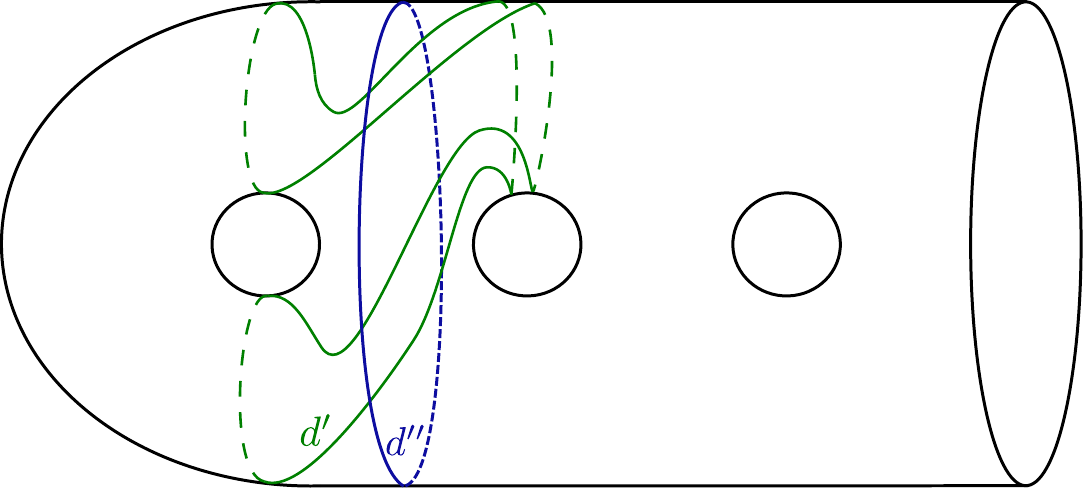}\\[0.5cm] 
 \includegraphics[scale=0.5]{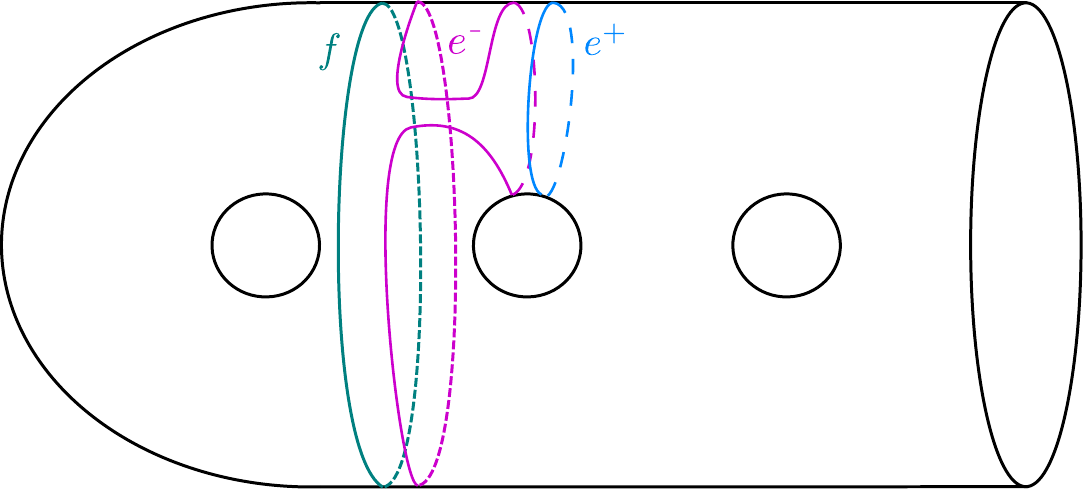}\qquad 
  \includegraphics[scale=0.5]{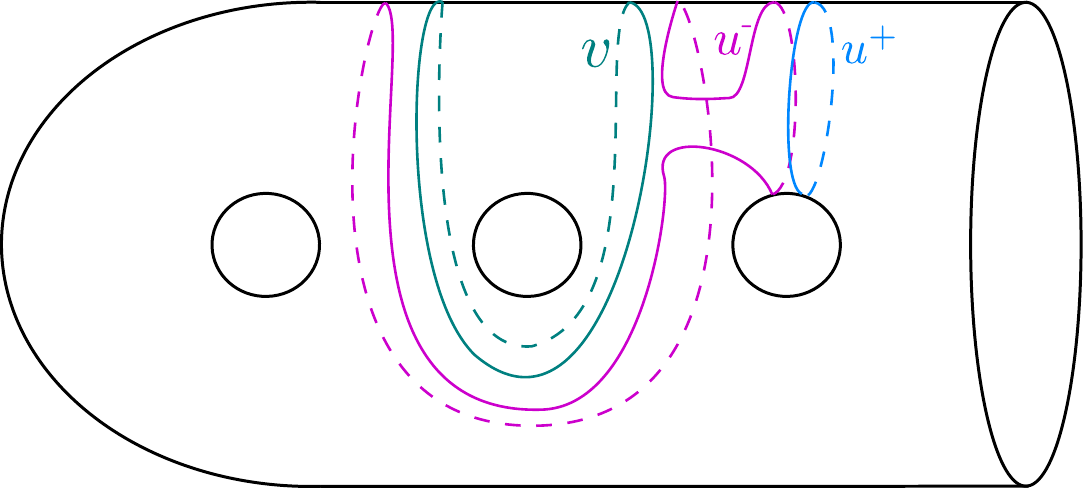}  
  \end{tabular}
\caption{The curves defining instances of $\varphi_1$, $\varphi_2$ and $\varphi_3$} \label{fig:more_curves}
\end{figure}
\end{remark}

\subsection{Complements} \label{subsec:complements}

We conclude by discussing the size of the torsion subgroup of the abelianized Johnson kernel. For that, we will give a lower bound 
on the size of the image of the map $R_{\operatorname{ab}}: \calK_{\operatorname{ab}} \to \calT_4(H^\Q)/\calT_4(H)$:
\begin{equation} \label{eq:ineq1}
\left\vert \operatorname{Tors} (\calK_{\operatorname{ab}} ) \right\vert \geq 
\left\vert  R_{\operatorname{ab}}\big(\operatorname{Tors} (\calK_{\operatorname{ab}} )\big) \right\vert
\end{equation}
Note that the above inequality may be strict. Indeed
we have the  inclusions of groups
$$
\{0\} \subset \frac{{\Gamma_4 \calI}\cdot \calK'}{\calK'}\subset \underbrace{\frac{\sqrt{\Gamma_4 \calI}}{\calK'}}_{=\operatorname{Tors} (\calK_{\operatorname{ab}} )} \subset\frac{{\calM[4]}}{\calK'}  \subset   \calK_{\operatorname{ab}}
$$
where the description of $\operatorname{Tors} \big(\calK_{\operatorname{ab}}\big)$   has been justified just after \eqref{eq:ker(j)}.   
Since $\tau_4$ maps $\Gamma_4\calI$ to $\calT_4(H) \subset \calT_4(H^\Q) $, we deduce from Lemma \ref{lem:R_M4} 
 that  $R_{\operatorname{ab}}$ vanishes on  ${\big({\Gamma_4 \calI}\cdot \calK'}\big)/{\calK'}$. 
 Hence, if $\Gamma_4 \calI$ is not included in $\calK'$ (which the authors ignore), then $R_{\operatorname{ab}}$
is not enough to detect all the torsion of~$\calK_{\operatorname{ab}}$. 

\begin{remark}
Actually, the authors do not even know whether $R_{\operatorname{ab}}$ has only $2$-torsion. \hfill $\blacksquare$
\end{remark}

To obtain a lower bound on the size of the image of $R_{\operatorname{ab}}$,
we use  Lemma \ref{lem:R_M4}, the $\calM$-equivariance property of $\tau_4$ and formula \eqref{eq:dos}:
\begin{equation} \label{eq:ineq2}
\left\vert  R_{\operatorname{ab}}\big(\operatorname{Tors} (\calK_{\operatorname{ab}} )\big) \right\vert
 \geq  \left\vert  R_{\operatorname{ab}}\big( \langle \varphi\rangle_{\calM} )\big) \right\vert
 =  \left\vert  \varpi \tau_4 \big( \langle \varphi\rangle_{\calM} )\big) \right\vert 
 =  \left\vert \big\langle [[a_1,a_2],a_3] \rangle_{ \operatorname{Sp}(H)} \right\vert 
\end{equation}
Here $ \langle \varphi\rangle_{\calM} $ is the $\calM$-submodule of $\calK_{\operatorname{ab}}$ generated by $\{\varphi\}$
and $ \big\langle [[a_1,a_2],a_3] \rangle_{ \operatorname{Sp}(H)}$ is the $\operatorname{Sp}(H)$-submodule of 
$\frakL_3\otimes \Z_2$ generated by $ [[a_1,a_2],a_3]$. 
(Here and in the sequel, we simply denote by $x$ the element $x\otimes 1$ of   $\frakL_3\otimes \Z_2$ defined by any $x\in \frakL_3$.)

\begin{lemma} \label{lem:ses}
We have a split short exact sequence of $\operatorname{Sp}(H)$-modules 
$$
\xymatrix{
0\ar[r] & H \otimes \Z_2 \ar[r]^-{[\omega,-]} & \frakL_3\otimes \Z_2   \ar[r]^-p  \ar@/^1pc/[l]^-\stigma &  \clo{\frakL}_3\otimes \Z_2  \ar[r]  & 0,
}
$$
where $p$ is the canonical projection and $\stigma: \frakL_3 \otimes \Z_2  \to  H \otimes \Z_2 $ is defined by 
$\stigma([[a,b],c])=\omega(b,c) \, a  + \omega(a,c)  \, b $.
\end{lemma}

\begin{proof}
By definition of $\clo{\frakL}$, we have a short exact sequence $0 \to H \to \frakL_3 \to \clo{\frakL}_3 \to 0$
and, since  $\clo{\frakL}_3$ is torsion-free, this sequence remains exact after tensorization with $\Z_2$. 
Regarding $ \frakL_3$ as a quotient of $H^{\otimes 3}$, we easily verify that $\stigma$ is well-defined.
Besides, we have $\stigma([\omega,h])=h $ for all $h\in H$, 
showing that the  short exact sequence  is split.
\end{proof}

\begin{lemma} \label{lem:kernel}
With the notations of Lemma \ref{lem:ses} and for $g\geq 3$, we have $\ker \stigma= \big\langle [[a_1,a_2],a_3] \rangle_{ \operatorname{Sp}(H)} $.
\end{lemma}

\begin{proof} Set $S:=  \big\langle [[a_1,a_2],a_3] \rangle_{ \operatorname{Sp}(H)} $.
The $\operatorname{Sp}(H)$-equivariance of $\stigma$ implies that  $S \subset \ker \stigma$.
To prove the converse inclusion, we consider the projection $q: \frakL_3 \otimes \Z_2 \to \ker \stigma$ corresponding to the split short exact sequence of Lemma \ref{lem:kernel};
specifically, we have 
$$
q(x) = x +[\omega,\stigma(x)], \quad \forall x\in H.
$$
Let $Z=\{a_1,\dots,a_g,b_1,\dots,b_g\}$ and, for all $z,z'\in Z$, let us write $z \perp z'$ if we have $\{z,z'\}=\{a_i,b_i\}$ for some $i$.
Clearly, for any $z,z',z''\in Z$ with $z'\neq z''$, we have 
$$
q([[z',z''],z])= \left\{\begin{array}{lll}
[[z',z''],z] & \hbox{ if } (z\not\perp z' \hbox{ and }  z\not\perp z'')  & \hbox{(i)}  \\
{[[z',z''],z]} +[\omega,z'']& \hbox{ if } z\perp z' &\hbox{(ii)} \\
{[[z',z''],z]} +[\omega,z']& \hbox{ if } z\perp z'' &\hbox{(iii)}.\\
\end{array}\right.
$$
Since  $\frakL_3 \otimes \Z_2$ is generated by the elements $[[z',z''],z]$, for all $z,z',z''\in Z$ with $z'\neq z''$,
we deduce that $\ker \stigma$ is generated by the above elements of types (i), (ii) and (iii). Hence,
we are reduced to show that any element of type (i) and any element of type (ii) belong to $S$.
For that, we will use the following elements of $\operatorname{Sp}(H)$:
\begin{itemize}
\item for $x\in H$, $T_x$ is the transvection defined by $T_x(h) = h + \omega(x,h)x$;
\item for any $r, s \in \{1,\dots,g\}$ with $r\neq s$, $E_{rs}$ exchanges $a_r$ (resp$.$ $b_r$) with $a_s$ (resp$.$ $b_s$)
and fixes all other elements of $Z$;
\item for each $r \in \{1,\dots,g\}$, $F_r$ maps $a_r$ (resp$.$ $b_r$) to $-b_r$ (resp$.$ $a_r$)  and fixes all other elements of~$Z$.
\end{itemize}

We start by considering the elements of type (i). They can be of the following forms:
\begin{eqnarray*}
[[a_i,a_j],a_k] &&  \hbox{with ($i,j,k$ pairwise disjoint) or ($k=j$ and $i\neq j$)},  \\
 {[[a_i,b_j],a_k]} &&   \hbox{with ($i,j,k$ pairwise disjoint) or ($k=i$ and $i\neq j$)  or ($j=i$ and $k\neq i$)}, \\
 {[[b_i,b_j],a_k]}  &&  \hbox{with ($i,j,k$ pairwise disjoint)},\\
 {[[b_i,b_j],b_k]}   && \hbox{with ($i,j,k$ pairwise disjoint) or ($k=j$ and $i\neq j$)},  \\
 {[[b_i,a_j],b_k]}  && \hbox{with ($i,j,k$ pairwise disjoint) or ($k=i$ and $i\neq j$) or ($j=i$ and $k\neq i$)}, \\
 {[[a_i,a_j],b_k]} && \hbox{with ($i,j,k$ pairwise disjoint)}.
\end{eqnarray*}
The last three forms can be derived from the first three forms by applying symplectic transformations of type $F_r$:
hence it suffices to consider the first three forms.
We have 
\begin{eqnarray*}
T_{b_1}\cdot  [[a_1,a_2],a_3] &=&  [[a_1,a_2],a_3]  +  [[b_1,a_2],a_3]\\
\hbox{and} \quad T_{b_2} \cdot [[b_1,a_2],a_3] &= &[[b_1,a_2],a_3]  + [[b_1,b_2],a_3],
\end{eqnarray*}
hence we obtain  that $ [[b_1,a_2],a_3] \in S$ and $[[b_1,b_2],a_3] \in S$. By using the transformations of type~$E_{rs}$, we deduce that 
 all $[[a_i,a_j],a_k]$, $ [[a_i,b_j],a_k]$ and $[[b_i,b_j],a_k]$ with $i,j,k$ {pairwise disjoint}, belong to $S$. Furthermore, we have
\begin{eqnarray*}
T_{a_2+a_3}\cdot [[a_1,a_2],b_3] &= &[[a_1,a_2],b_3] +  [[a_1,a_2],a_2] +  [[a_1,a_2],a_3] \\
\hbox{and} \qquad\quad  F_1 \cdot [[a_1,a_2],a_2] &=& [[b_1,a_2],a_2],
\end{eqnarray*}
hence we obtain that  $ [[a_1,a_2],a_2] \in S$ and $[[b_1,a_2],a_2] \in S$. 
By using the transformations of type~$E_{rs}$, we deduce that $[[a_i,a_j],a_j]\in S$ and $ [[a_i,b_j],a_i]\in S$ for all $i\neq j$.
Finally, we have 
$$
T_{b_1} \cdot [[a_1,a_2],a_1] =[[a_1,a_2],a_1] + [[b_1,a_2],b_1]   + \underbrace{[[b_1,a_2],a_1]  + [[a_1,a_2],b_1]}_{=[[a_1,b_1],a_2]}
$$
and we deduce that $[[a_1,b_1],a_2] \in S$: hence we have $ {[[a_i,b_i],a_k]} \in S$ for all $i\neq k$. 
Thus we have checked that all elements of type (i) belong to $S$.

We now consider the elements of type (ii). They can be of the following forms:
$$
 q([[a_i,b_i],b_i]),  \quad q([[b_i,a_j],a_i])  \ \hbox{with $i\neq j$},   \quad q([[a_i,a_j],b_i]) \ \hbox{with $i\neq j$},
$$
$$
 q([[b_i,a_i],a_i]), \quad  q([[a_i,b_j],b_i])   \ \hbox{with $i\neq j$}, \quad q([[b_i,b_j],a_i]) \ \hbox{with $i\neq j$}.
$$
The last three forms can be derived from the first three forms by applying symplectic transformations of type $F_r$:
hence it suffices to consider the first three forms.
In the sequel, we denote by $\equiv$ the congruence modulo $S$. First, note that
$$
q([[a_i,b_i],b_i])= [[a_i,b_i],b_i] + [\omega,b_i] \equiv  [[a_i,b_i],b_i]  +  [[a_i,b_i],b_i]  =0
$$
where the congruence follows from the consideration of elements of type (i). 
Second, we have 
$$
q([[b_i,a_j],a_i]) = [[b_i,a_j],a_i] + [\omega,a_j] \equiv [[b_i,a_j],a_i] + [[a_j,b_j],a_j].
$$
Let $G_{ij}$ be the symplectic transformation of $H$ that maps $a_i$ to $a_i+a_j$, $b_i$ to $b_i$, $a_j$ to $-a_j$, $b_j$ to $b_i-b_j$,
and fixes all other elements of $S$. Observe that
\begin{eqnarray*}
G_{ij} \cdot [[a_j,b_j],a_i] &=& [[a_j,b_i+b_j],a_i+a_j] \\
&=& [[a_j,b_i],a_i] + [[a_j,b_i],a_j]  + [[a_j,b_j],a_i]  + [[a_j,b_j],a_j] \\
&\equiv &  [[a_j,b_i],a_i]  + [[a_j,b_j],a_j] 
\end{eqnarray*}
where the congruence follows from the consideration of elements of type (i). Thus we deduce that  $q([[b_i,a_j],a_i]) \in S$.
Third, we have
$$
q([[a_i,a_j],b_i]) = [[a_i,a_j],b_i] + [\omega, a_j] \equiv  [[a_i,a_j],b_i] + [[a_j,b_j], a_j] \equiv [[a_j,b_i],a_i]  + [[a_j,b_j],a_j]
$$
where the last conguence follows from the Jacobi identity: hence we are back to the second form of elements of type (ii).
\end{proof}

We can now conclude by giving lower bounds on the sizes   of the torsion parts  of the abelianized Johnson kernels.

\begin{proposition} \label{prop:lower_bounds}
Assume that $g\geq 6$. The cardinality of $\operatorname{Tors} (\calK_{\operatorname{ab}} ) $ is at least $2^{\frac{8}{3}(g^3-g)}$,
and the  cardinality of $ \operatorname{Tors} (\clo{\calK}_{\operatorname{ab}} ) $ is at least $2^{\frac{1}{3}(g^3-4g)}$.
\end{proposition}

\begin{proof}
We deduce from Lemma \ref{lem:kernel} that  $\big\langle [[a_1,a_2],a_3] \rangle_{ \operatorname{Sp}(H)}$
is isomorphic to $(\frakL_3(H)/H)\otimes \Z_2$ as a $\Z_2$-vector space.
Hence, thanks to \eqref{eq:ineq1} and \eqref{eq:ineq2}, we obtain the lower bound 
$$
\left\vert \operatorname{Tors} (\calK_{\operatorname{ab}} )\right\vert  \geq 2^{\operatorname{rk} \frakL_3(H) - \operatorname{rk} H}
$$ 
in the bordered case. In the closed case, it follows from diagram \eqref{eq:RR}  and Remark \ref{rem:about_L} that
$$
\left\vert \operatorname{Tors} (\clo{\calK}_{\operatorname{ab}} )\right\vert  \geq 2^{\operatorname{rk} \frakL_3(A) - \operatorname{rk}(A)}
$$
where $A$ is the quotient $H/\langle b_1,\dots,b_g\rangle$. 
\end{proof}

\begin{remark}
It is expected that the lower bounds of Proposition \ref{prop:lower_bounds} are far from optimal (at least in the closed case). \hfill $\blacksquare$
\end{remark}

\bigskip

\appendix

\section{Computation of the symplectic logansion} \label{sec:code}

\begin{lstlisting}[language=Python]
# Choose the genus and the nilpotency class

g=3
N=3


# The free Lie algebra on 2g generators a1,...,ag,b1,...,bg of nilpotency class N

L=LieAlgebra(QQ,2*g,step=N)
a=[L.gen(i) for i in range(g)]
b=[L.gen(i+g) for i in range(g)]


# Values of the symplectic logansion "theta" up to order N

theta_a = [ a[i] - (1/2) * L[a[i],b[i]]
+ (1/12) * L[L[a[i],b[i]],b[i]] -(1/24)* L[a[i],L[a[i],L[a[i],b[i]]]]
- (1/2) * L[ sum( L[a[j],b[j]] for j in range(i) ) , a[i] ]
+ (1/4) * L[ sum( L[a[j],b[j]] for j in range(i) ) , L[a[i],b[i]] ] 
for i in range(g) ]

theta_b = [ b[i] - (1/2) * L[a[i],b[i]] + (1/4) * L[L[a[i],b[i]],b[i]]
+ (1/12) * L[a[i],L[a[i],b[i]]] -(1/24)* L[L[L[a[i],b[i]],b[i]],b[i]]
+ (1/2) * L[ b[i], sum( L[a[j],b[j]] for j in range(i) ) ]
+ (1/4) * L[ sum( L[a[j],b[j]] for j in range(i) ) , L[a[i],b[i]] ]  
for i in range(g) ]


# Computation of theta from a string such a 'a1+b2-a1-' which encodes an element of the fundamental group

def theta(lis):
	res = L(0)
	for j in range(len(lis)/3):
            index = int(lis[3*j+1])-1
            if [lis[3*j],lis[3*j+2]]==['a','+']: res = L.bch(res,theta_a[index])
            if [lis[3*j],lis[3*j+2]]==['a','-']: res = L.bch(res,-theta_a[index])
            if [lis[3*j],lis[3*j+2]]==['b','+']: res = L.bch(res,theta_b[index])
            if [lis[3*j],lis[3*j+2]]==['b','-']: res = L.bch(res,-theta_b[index])
	return res


# Creation of strings corresponding to commutators in the fundamental group
	
def invert(lis):
	w=''
	for j in range(len(lis)/3):
		if lis[3*j+2]=='+':
			w = lis[3*j]+lis[3*j+1]+'-'+w
		else:
			w = lis[3*j]+lis[3*j+1]+'+'+w
	return w

def comm(a,b):
	return a+b+invert(a)+invert(b)
	
    
# Checks that the logansion theta is symplectic up to degree N

boundary = ''
for i in range(g): boundary = boundary + 'b' + str(i+1) + '-' + 'a' + str(i+1) + '+'  + 'b' + str(i+1) + '+' + 'a' + str(i+1) + '-'  

omegatilde = theta(boundary)
omega = sum(L[a[i],b[i]] for i in range(g))

if omegatilde==omega:
    print('OK: the expansion is symplectic up to order '+str(N)) 
else:
    print('Warning: the given expansion is not symplectic up to the given order !')


#  Display of an element of the nilpotent free Lie algebra
	
import sage.combinat.words.lyndon_word as lyndon_word

def transform(lis):
	if lis in [1..g]:
		return 'a'+str(lis)
	if lis in [(g+1)..(2*g)]:
		return 'b'+str(lis-g)
	if len(lis)==1:
		return transform(lis[0])
	else:
		return '[' + (transform([lis[0]])) + ',' + (transform(lis[1])) +']'

def display(x):
	Llist = (L.basis()).list()
	LW = LyndonWords(2*g,1).list()
	for i in [2..N]:
		LW = LW + LyndonWords(2*g,i).list()
	if x == 0:
		return '0'
	else:
		if x != x.leading_monomial():
			z=x.leading_coefficient()
			return display(x-z*x.leading_monomial())+'+('+str(z)+')*'+display(x.leading_monomial())
		else:
			for j in range(len(Llist)):
				if x == Llist[j]:
					return transform(lyndon_word.standard_bracketing(LW[j]))
\end{lstlisting}

\end{document}